\documentclass[11pt]{article}
\usepackage[utf8]{inputenc}
\usepackage{amsmath}
\usepackage{amssymb}
\usepackage{amsthm}
\usepackage{fullpage}
\usepackage{graphicx}
\usepackage{hyperref}
\usepackage{color}
\usepackage[title]{appendix}
\usepackage[font=small,width=\textwidth]{caption}
\usepackage{enumerate}
\usepackage{array}

\newtheorem{thm}{Theorem}[section]
\newtheorem{prop}[thm]{Proposition}
\newtheorem{lem}[thm]{Lemma}
\newtheorem{coro}[thm]{Corollary}

\newcommand{\tdef}[1]{\emph{#1}}
\newcommand{\EX}{\mathbb{E}}
\newcommand{\cH}{\mathcal{H}}
\newcommand{\cK}{\mathcal{K}}
\newcommand{\cJ}{\mathcal{J}}
\newcommand{\cT}{\mathcal{T}}
\newcommand{\cB}{\mathcal{B}}
\newcommand{\eps}{\varepsilon}
\newcommand{\pg}{p_0}

    \makeatletter
    \let\@fnsymbol\@arabic
    \makeatother

\author{Oliver Cooley\thanks{Supported by Austrian Science Fund (FWF) \& German Research Foundation (DFG): I3747}\! ,  Wenjie Fang \thanks{Supported by Austrian Science Fund (FWF): I2309 \& P27290}\! , Nicola Del Giudice\thanks{Supported by Austrian Science Fund (FWF): W1230II}\! ,  Mihyun Kang \footnotemark[2] \footnotemark[3] \footnotemark[4]\\
\small Graz University of Technology\\
\small Institute of Discrete Mathematics\\
\small 8010 Graz, Austria\\
\small\tt \{cooley,fang,delgiudice,kang\}@math.tugraz.at}

\title{Subcritical random hypergraphs, high-order components, and hypertrees\footnote{An extended abstract of this work has been accepted by the Proceedings of Analytic Algorithmics and Combinatorics (ANALCO19).}}
\begin{document}

\maketitle

\abstract{One of the central topics in the theory of random graphs deals with the phase transition in the order of the largest components. In the binomial random graph $\mathcal{G}(n,p)$, the threshold for the appearance of the unique largest component (also known as the giant component) is $p_g = n^{-1}$. More precisely, when $p$ changes from $(1-\eps)p_g$ (subcritical case) to $p_g$ and then to $(1+\eps)p_g$ (supercritical case) for $\eps>0$, with high probability the order of the largest component increases smoothly from $O(\eps^{-2}\log(\eps^3 n))$ to $\Theta(n^{2/3})$ and then to $(1 \pm o(1)) 2 \eps n$. Furthermore, in the supercritical case, with high probability the largest components except the giant component are trees of order $O(\eps^{-2}\log(\eps^3 n))$, exhibiting a structural symmetry between the subcritical random graph and the graph obtained from the supercritical random graph by deleting its giant component.

As a natural generalisation of random graphs and connectedness, we consider the binomial random $k$-uniform hypergraph $\mathcal{H}^k(n,p)$ (where each $k$-tuple of vertices is present as a hyperedge with probability $p$ independently) and the following notion of high-order connectedness. Given an integer $1 \leq j \leq k-1$, two sets of $j$ vertices are called \emph{$j$-connected} if there is a walk of hyperedges between them such that any two consecutive hyperedges intersect in at least $j$ vertices. A $j$-connected component is a maximal collection of pairwise $j$-connected $j$-tuples of vertices. Recently, the threshold for the appearance of the giant $j$-connected component in $\mathcal{H}^k(n,p)$ and its order were determined.
In this article, we take a closer look at the subcritical random hypergraph. We determine the structure, order, and size of the largest $j$-connected components, with the help of a certain class of ``hypertrees'' and related objects. In our proofs, we combine various probabilistic and enumerative techniques, such as generating functions and couplings with branching processes. Our study will pave the way to establishing a symmetry between the subcritical random hypergraph and the hypergraph obtained from the supercritical random hypergraph by deleting its giant $j$-connected component.
}

\section{Introduction}

\subsection{Motivation} \label{sec:intro:motivation}
One of the most prominent results on random graphs is the so-called {\em phase transition} in the order of the largest components, first discussed by Erd\H{o}s and R\'enyi in their seminal work \cite{ErdosRenyi60}: a small change in the edge density around the critical value drastically alters the structure and order of the largest components. Their result was improved for example by Bollob\'as \cite{Bollobas84} and \L{}uczak \cite{Luczak90} and is often stated for the binomial random graph $\mathcal G(n,p)$, a graph with vertex set $[n]:=\{1,\ldots, n\}$ in which each pair of vertices is present as an edge with probability $p$ independently. Throughout the paper, $\log$ denotes the natural logarithm and we say that an event holds {\em with high probability} ({\em whp} for short) if the probability that it holds tends to $1$ as $n \to \infty$. 
 
\begin{thm}[\cite{Bollobas84,BollobasBook,ErdosRenyi60,Luczak90}] \label{thm:giantgraph}
Let $0<\eps<1 $ be a constant or a function in $n$ satisfying $\eps \rightarrow 0$ and $\ell:=\eps^3 n \rightarrow \infty$. For each $i\in \mathbb N$, let $C_i=C_i(\mathcal G(n,p))$ denote the number of vertices in the $i$-th largest component in $\mathcal{G}(n,p)$.
\begin{itemize}
\item[(1)] If $p = (1-\eps)n^{-1}$, then whp every component is either a tree or unicyclic, and for every constant $i\ge 1$, the $i$-th largest component is a tree. Furthermore, for any function $\omega=\omega(n) \rightarrow \infty$, whp
\[
\left|C_1-\alpha^{-1} \left(\log \ell - \frac{5}{2}\log \log \ell \right) \right| \le \omega \eps^{-2},
\]
where $\alpha = - \eps - \log (1-\eps)= \eps^2 / 2 + O(\eps^3)$ and $C_i= (2+o(1))\eps^{-2} \log  (\ell)$ for $i\ge 2$.
\item[(2)] If $p = (1+\eps)n^{-1}$, then whp the largest component contains at least two cycles and
\[ 
C_1=(1 + o(1)) 2 \eps n.
\] 
Furthermore, for every $i \ge 2$, whp the $i$-th largest component is a tree with $C_i = (2+o(1))( \eps^{-2} \log \ell)$ and in particular, for any function $\omega=\omega(n) \rightarrow \infty$, whp
\[
\left|C_2-\beta^{-1} \left(\log \ell - \frac{5}{2}\log \log  \ell \right) \right| \le \omega \eps^{-2},
\]
where $\beta = \eps - \log (1+\eps)= \eps^2 / 2 + O(\eps^3)$.
\end{itemize}
\end{thm}
Note that in the supercritical random graph (\textit{i.e.}\ when $p = (1+\eps)n^{-1}$), the largest component is substantially larger than the second largest component and therefore it is also called the \emph{giant} component. Theorem~\ref{thm:giantgraph} displays a \emph{symmetry} between the structure of the subcritical random graph and the supercritical random graph with the giant component removed.

Since this result, various models of random graphs have been introduced and analysed for similar phenomena. For instance, considerable attention has been paid to random regular graphs (see \textit{e.g.}\ \cite{BollobasBook,Wormald99}  for an overview). As for $\mathcal{G}(n,p)$, this random graph model is ``homogeneous'', meaning that all vertices are equivalent, and whp all vertex degrees lie within a small range. 
In contrast, many random models for ``real-world'' graphs show highly inhomogeneous behaviour with a large spread of degrees, whose distributions often follow a power law. Thus, also the phase transition in inhomogeneous random graphs such as scale-free graphs \cite{BollobasRiordan03} and distance graphs \cite{AjaziNapolitanoTurova17} has been investigated \cite{BollobasJansonRiordan07}, with a wide analysis of the subcritical case (see \textit{e.g.}\ \cite{Janson08,Turova11}).

Higher-dimensional analogues of random graphs and their phase transitions have also drawn particular attention. The most commonly studied higher-dimensional analogue of $\mathcal G(n,p)$ is the binomial random $k$-uniform hypergraph $\cH^k(n,p)$ defined below. Amongst other properties, vertex-connectedness
\cite{AndriamaRavelomanana05,BehrischCojaOghlanKang10b,BCOK14,BollobasRiordan12c,BollobasRiordan17,
KaronskiLuczak96,KaronskiLuczak97,Poole15,
RavelomananaRijamamy2006,SchmidtShamir85}
and high-order connectedness (also known as $j$-tuple-connectedness)
\cite{CKKgiant,CooleyKangPerson18} of $\cH^k(n,p)$  have been extensively studied and will be discussed in more detail in Section \ref{sec:intro:relwork}.
In parallel, Linial and Meshulam instigated research into random simplicial complexes,
which have since also been extensively studied~\cite{KahlePittel16,LinialMeshulam06,LinialPeled16,MeshulamWallach08}.

Before stating our results, we introduce the necessary concepts.
Let $k \geq 2$ and $1 \leq j \leq k-1$ be integers. 
A \tdef{$k$-uniform hypergraph} $H$ is a pair $H=(V,E)$, where $V$ is the set of \tdef{vertices} and $E\subseteq \binom{V}{k}$, a collection of $k$-element subsets of $V$, is the set of \tdef{hyperedges}. An $\ell$-element subset of $V$ is called an \tdef{$\ell$-set of $V$} (or \tdef{$\ell$-set} for short). A pair $\{J_1,J_2\}$ of $j$-sets are called \tdef{$j$-tuple-connected} (\tdef{$j$-connected} for short) in $H$ if there is a sequence of hyperedges $K_1,\ldots, K_m$ such that $J_1 \subset K_1$, $J_2 \subset K_m$ and $|K_i \cap K_{i+1}| \geq j$ for all $1\leq i \leq m-1$. Additionally, any $j$-set is always $j$-connected to itself. The \tdef{$j$-connected components} (\tdef{$j$-components} for short) of $H$ are equivalence classes of the relation $\sim_j$ defined by $J_1 \sim_j J_2$ if and only if $J_1$ and $J_2$ are $j$-connected. In other words, a $j$-component is a maximal collection of $j$-sets that are pairwise $j$-connected. Given a $j$-component $\cJ$, it also comes naturally with a set of hyperedges $\cK_{\cJ}$, which are the hyperedges containing any of the $j$-sets in $\cJ$. In a slight abuse of terminology, we say that the $j$-component $\cJ$ also contains the hyperedges $\cK_{\cJ}$. The \tdef{order} of a $j$-component denotes the number of $j$-sets it contains, and the \tdef{size} of a $j$-component denotes the number of hyperedges it contains. A \tdef{hypertree $j$-component}\label{term:hypertree} (a \tdef{hypertree} for short) is a $j$-component that contains as many $j$-sets as possible {\em given its size}, \textit{i.e.}\ if it has size $s$ and order $t$, then $t=1+\left(\binom{k}{j}-1\right)s$. The case $k=2$ and $j=1$ corresponds to the classical concepts in a graph.

We denote by $\cH^k(n,p)$ the \tdef{random $k$-uniform hypergraph} with vertex set $[n]$, in which each hyperedge is present with probability $p$ independently. When parameters are clear from the context, we use $\cH$ as a shorthand 
for $\cH^k(n,p)$. The following higher-dimensional analogue of the random graph phase transition for $\cH$ and $j$-connectedness was obtained in \cite{CKKgiant, CooleyKangPerson18}.
\begin{thm}[\cite{CKKgiant, CooleyKangPerson18}]\label{thm:gianthyper}
Given integers $k\geq 2$ and $1\leq j \leq k-1$, let $\eps=\eps(n)>0$ satisfy $\eps\rightarrow 0$, $\eps^3 n^j\rightarrow \infty$ and $\eps^2 n^{1-2\delta} \rightarrow \infty$, as $n\to\infty$, for some constant $\delta >0$. Let 
\[
\bar{p}_0 := \left( \binom{k}{j} - 1 \right)^{-1} \binom{n}{k-j}^{-1}.
\]
\begin{itemize}
\item[(1)]  If $p=(1-\eps)\bar{p}_0$, then whp all $j$-components of $\cH^k(n,p)$ have order $O(\eps^{-2}\log n)$.
\item[(2)] If $p=(1+\eps)\bar{p}_0$, then whp the order of the largest $j$-component of $\cH^k(n,p)$ is $(1\pm o(1))\frac{2\eps}{\binom{k}{j}-1}\binom{n}{j}$, while all other $j$-components have order $o(\eps n^j)$.
\end{itemize}
\end{thm}

The aim of this paper is to strengthen Theorem~\ref{thm:gianthyper} in view of Theorem~\ref{thm:giantgraph} by taking a closer look at the {\em subcritical} case and addressing the following natural questions.
\begin{itemize}
\item[(a)] What are the precise asymptotic order and size of the largest $j$-component of $\cH^k(n,p)$? 
\item[(b)] What does the largest $j$-component look like? Is it whp a hypertree or some other more complex structure? 
\end{itemize}

\subsection{Main result} \label{sec:intro:mainresult}

In Theorems~\ref{thm:giantgraph} and~\ref{thm:gianthyper}, the order of a $j$-component  was studied.
In this paper, we concentrate on the size of a $j$-component, \textit{i.e.}\ the number of hyperedges
it contains. Therefore, whenever we talk about the \emph{$i$-th largest $j$-component},
the ranking is determined by the size
rather than the order (\textit{i.e.}\ the number of $j$-sets it contains).
Observe that in the range of the hyperedge probability in our study, whp the order and size of the largest $j$-component only differ roughly by a multiplicative constant $c_0 = \binom{k}{j} - 1$. As a consequence, we also derive the order of the largest $j$-components. 

Before stating our main result, we introduce the following notation. Let $(X_n)_{n\ge 0}$ be a sequence of random variables and $(a_n)_{n\ge 0}$ a sequence of positive real numbers. We say that $X_n=O_p(a_n)$ if for every $\gamma > 0$ there exist $C_\gamma$ and $n_0 \in \mathbb{N}$ such that $\Pr \left(|X_n|\leq C_\gamma\right) > 1 - \gamma$ for every $n\geq n_0$. It is easy to see that $X_n=O_p(a_n)$ if and only if $\Pr \left( |X_n| \leq \omega a_n \right) \rightarrow 1$ for every function $\omega=\omega(n) \rightarrow \infty$.
\begin{thm} \label{thm:mainthm}
Given integers $k \geq 2$ and $1 \leq j \leq k-1$, and $\eps=\eps(n)$ with $0 < \eps < 1$, $\eps^4 n^j \to \infty$ and $\eps^2 n^{k-j} (\log n)^{-1} \to \infty$, let 
\[ c_0=\binom{k}{j}-1  \;\; \mathrm{and} \;\; \pg = c_0^{-1} \binom{n-j}{k-j}^{-1}. \label{par:c0andp0} \] 
For $i \in \mathbb{N}$, let $\mathcal{L}_i=\mathcal{L}_i(\cH^k(n,p))$ be the $i$-th largest $j$-component of $\cH^k(n,p)$, and $L_i$\label{var:Li} its size. If $p = (1-\eps)\pg$, then for any fixed $m \in \mathbb{N}$, whp for any $1 \leq i \leq m$, $\mathcal{L}_i$ is a hypertree component, with size
\[
L_i = \delta^{-1}\left(\log \lambda - \frac{5}{2}\log\log \lambda + O_p(1)\right),
\]
where $\delta = -\eps - \log (1-\eps)= \eps^2 / 2 + O(\eps^3)$\label{par:delta} and $\lambda = \eps^3 \binom{n}{j}$\label{par:lambda}. 
\end{thm}

We note that the critical probability $p_0$ differs from  $\bar{p}_0$ (defined in Theorem~\ref{thm:gianthyper}) by a factor of $1+O(n^{-1})$. This is because we analyse a range closer to criticality, which requires a more precise value of $p_0$.

We also note that the coefficient $-5/2$ before the $\log\log\lambda$ factor in Theorem~\ref{thm:mainthm} is the same as that in Theorem~\ref{thm:giantgraph} for graphs, and arises from the universal asymptotic behaviour of various families of labelled trees, \textit{i.e.}\ connected acyclic graphs. More precisely, the asymptotic number of trees on $t$ vertices in such a family has the form $c \cdot t! \gamma^t t^{-5/2}$, with $c$ and $\gamma$ depending on the precise nature of the family. The proofs of both Theorem~\ref{thm:giantgraph} and Theorem~\ref{thm:mainthm} involve asymptotic counting of such families of trees, and the coefficient $-5/2$ comes from the common polynomial factor $t^{-5/2}$. In the case when the trees are rooted, which is more commonly considered, the exponent $-5/2$ would become $-3/2$ (see \cite[Section~VII.3]{FlajoletSedgewickBook}) -- the extra factor of $t$ comes from the choice of the root.

Since Theorem~\ref{thm:mainthm} states that whp the largest $j$-components are \emph{hypertrees}, as a corollary we can determine their order and obtain the following result.

\begin{coro} \label{cor:comporder}
    Let $k,j,\eps,p,\delta,\lambda,\mathcal{L}_i$ be given as in Theorem~\ref{thm:mainthm}.
    For $i \in \mathbb{N}$, let $M_i=M_i\left( \cH^{k}(n,p) \right)$ be the order of $\mathcal{L}_i$.
    Then for any fixed $m \in \mathbb{N}$, whp for any $1\leq i \leq m$ we have
    \[
    M_i = c_0 \delta^{-1}\left(\log \lambda - \frac{5}{2}\log\log \lambda + O_p(1)\right).
    \]
   
\end{coro} 

Note that when $k=2$ and $j=1$ (\textit{i.e.}\ the graph case) Corollary~\ref{cor:comporder} gives exactly $M_1=\delta^{-1} \left( \log \lambda - \frac{5}{2} \log \log \lambda + O_p(1) \right)$, as stated in Theorem~\ref{thm:giantgraph}-(1).

\subsection{Related work} \label{sec:intro:relwork}

The case $j=1$ of $j$-tuple-connectedness corresponds to vertex-connectedness of $\cH^k(n,p)$ and is the most studied among the higher-dimensional analogues of the phase transition in $\mathcal{G}(n,p)$. Enumeration results for the asymptotic number of $1$-connected $k$-uniform hypergraphs were obtained by Karo\'nski and {\L}uczak \cite{KaronskiLuczak97}, later improved by Andriamampianina and Ravelomanana \cite{AndriamaRavelomanana05} via enumerative techniques. 
The threshold for the emergence of the giant $1$-component, \textit{i.e.}\ $p=(k-1)^{-1} \binom{n-1}{k-1}^{-1}$, was first determined by Schmidt-Pruzan and Shamir \cite{SchmidtShamir85}. 
Subsequently, Karo\'nski and {\L}uczak \cite{KaronskiLuczak02} studied the distribution of the order of the largest component in the early supercritical regime. The studied range was extended by Ravelomanana and Rijamamy \cite{RavelomananaRijamamy2006}, although they only computed the expected order of the largest component and not its distribution. 
Behrisch, Coja-Oghlan, and Kang \cite{BehrischCojaOghlanKang10b,BCOK14} provided central and local limit theorems for $p(k-1)\binom{n-1}{k-1}> 1+\eps$, with $\eps>0$ arbitrarily small but fixed, while more recently Bollob\'as and Riordan \cite{BollobasRiordan12c} showed that the distribution of the order of the largest component tends to a normal distribution for every $\eps = \omega(n^{-1/3})$. 

Vertex-connectedness was also studied for a model of non-uniform random hypergraphs, in which the probability for a hyperedge of size $t$ to belong to the hypergraph depends on a parameter $\omega_t$.
In particular, de Panafieu \cite{dePanafieu15} determined the critical value at which the first \emph{complex} component (\textit{i.e.}\ connected component with more than one cycle) appears.

In contrast, the case $j\geq2$ is not yet well-understood. Cooley, Kang, and Person \cite{CooleyKangPerson18} determined the threshold $\bar{p}_0$ for the appearance of the giant $j$-component and subsequently Cooley, Kang, and Koch \cite{CKKgiant} refined this result by determining the asymptotic order of the largest $j$-component after the threshold and by showing that the second largest component has much smaller order (see Theorem~\ref{thm:gianthyper}). However, the subcritical regime was not analysed, which is the aim of this paper.

Moreover, the analysis of the vertex-connectedness case shows a symmetry property: the supercritical hypergraph with the giant component removed behaves as a subcritical hypergraph (with slightly modified parameters).  It is not immediately clear that such a behaviour should also hold in the higher-dimensional case ($j\geq 2$) and we believe that our study will help to obtain results in this direction.

\section{Proof of Theorem~\ref{thm:mainthm}}

In order to prove Theorem~\ref{thm:mainthm}, we bound the size of the largest $j$-component from above and below: we prove that whp there is no $j$-component of size larger than the claimed value (Lemma~\ref{lem:upper-bound}) and that for any size smaller than the claimed value, there is at least one $j$-component (and indeed a hypertree component) with larger size (Lemma~\ref{lem:lower-bound-second-moment}).

In Section~\ref{sec:search}, we define the \emph{component search process} which explores a $j$-component starting from a $j$-set, and a \emph{two-type branching process} which gives an upper coupling on the search process (Lemma~\ref{lem:branchcoupling}). We will make use of these processes to obtain the following bound on sizes of $j$-components:
\begin{lem} \label{lem:upper-bound}
Let $k,j,\eps,p,\delta,\lambda$ be given as in Theorem~\ref{thm:mainthm}, and $K(n)\rightarrow \infty$. Whp, $\cH^k(n,p)$ contains no $j$-component of size larger than
\[ \delta^{-1}\left(\log \lambda - \frac{5}{2} \log \log \lambda + K(n)\right).\]
\end{lem}

Lemma~\ref{lem:upper-bound} will be proved in Section~\ref{sec:upper}. To this end, we bound from above (and below) the number of possible instances of the two-type branching process of a certain size (Lemma~\ref{lem:branchingconfigs}), which are the so-called \emph{rooted labelled two-type trees}. 
Using the coupling argument of Lemma~\ref{lem:branchcoupling}, we obtain an upper bound on the expected number of $j$-components whose size is  larger than a fixed value. 
We conclude by applying Markov's inequality to prove that if we run (at most) $\binom{n}{j}$ component search processes (each one starting from a different $j$-set), whp no component of size larger than the claimed value will be discovered.

In Section~\ref{sec:lower}, we obtain the lower bound on the size of the largest $j$-component and show that it is indeed a hypertree.

\begin{lem} \label{lem:lower-bound-second-moment}
  Let $k,j,\eps,p,\delta,\lambda$ be given as in Theorem~\ref{thm:mainthm}, and $K(n) \to \infty$. For any constant $m \in \mathbb{N}$, whp the largest $m$ components of $\cH^k(n,p)$ are hypertree components, each of size at least
  \[ \delta^{-1} \left( \log \lambda - \frac{5}{2} \log \log \lambda - K(n)\right). \]
\end{lem}

To prove Lemma~\ref{lem:lower-bound-second-moment}, in Section~\ref{sec:lower} we introduce \emph{wheels}, which are a higher-dimensional analogue of cycles in graphs, and bound their number from above  (Lemma~\ref{lem:wheelconfigs}). 
This will allow us to prove that for a certain range of $s$, most of the instances of our two-type branching process of size $s$ correspond to \emph{hypertree components} (Lemma~\ref{lem:bound-wheels}). Thus, we derive the desired lower bound by estimating the expected number of $j$-sets in large hypertree components (Lemma~\ref{lem:lower-bound-j}) and by a second moment argument. Finally, we show that whp in the considered range all components of sufficiently large size are hypertrees (Lemma \ref{lem:non-hypertree-bound}). 

For the reader's convenience, a glossary of some of the most important terminology and notation defined in the paper can be found in Appendix~\ref{ap:glossary}.
\begin{proof}[Proof of Theorem~\ref{thm:mainthm}]
Observe that Theorem~\ref{thm:mainthm} follows directly from Lemmas~\ref{lem:upper-bound} and~\ref{lem:lower-bound-second-moment}.
\end{proof}

\section{Search process and branching process} \label{sec:search}

Throughout the paper, we let $k$, $j$, $\eps$, $p$, $\delta$, $\lambda$, $c_0$, $p_0$ be as given in Theorem~\ref{thm:mainthm}. For positive integers $n\geq k$, we write $n_{(k)}$ for the \tdef{falling factorial} $n_{(k)}:= n(n-1)(n-2)\cdots (n-k+1) = n!/(n-k)!$.

We also introduce auxiliary two-type graphs. A \tdef{two-type graph}\label{term:twotypegr}
is a \emph{connected} bipartite graph on a set of vertices of type $k$ and a set of vertices of type $j$, where each vertex of type $k$ is connected to exactly $\binom{k}{j}$ vertices of type $j$. We only consider two-type graphs with at least one vertex of type $k$. A \tdef{labelled two-type graph}\label{term:labtwotypgr} with label set $[n]$ is a two-type graph with labels on its vertices, where vertices of type $j$ are labelled by $j$-sets of $[n]$ and vertices of type $k$ by $k$-sets of $[n]$ in such a way that to each of the $\binom{k}{j}$ vertices of type $j$ connected to a given vertex of type $k$ with label $K\in \binom{[n]}{k}$ is assigned a distinct $j$-set $J \subset K$ as label. 
There is a natural bijection between the set of labelled two-type graphs with label set $[n]$ whose
vertices all have distinct labels and the set of
\emph{possible components}, \textit{i.e.}\ pairs $(\cJ,\cK)$ which could potentially be the families
of $j$-sets and $k$-sets in a $j$-component of some hypergraph on vertex set $[n]$.
\emph{Acyclic} two-type graphs are called \tdef{two-type trees}\label{term:twotyptree}.

\begin{figure}
  \centering
  \includegraphics[page=1,width=0.8\textwidth]{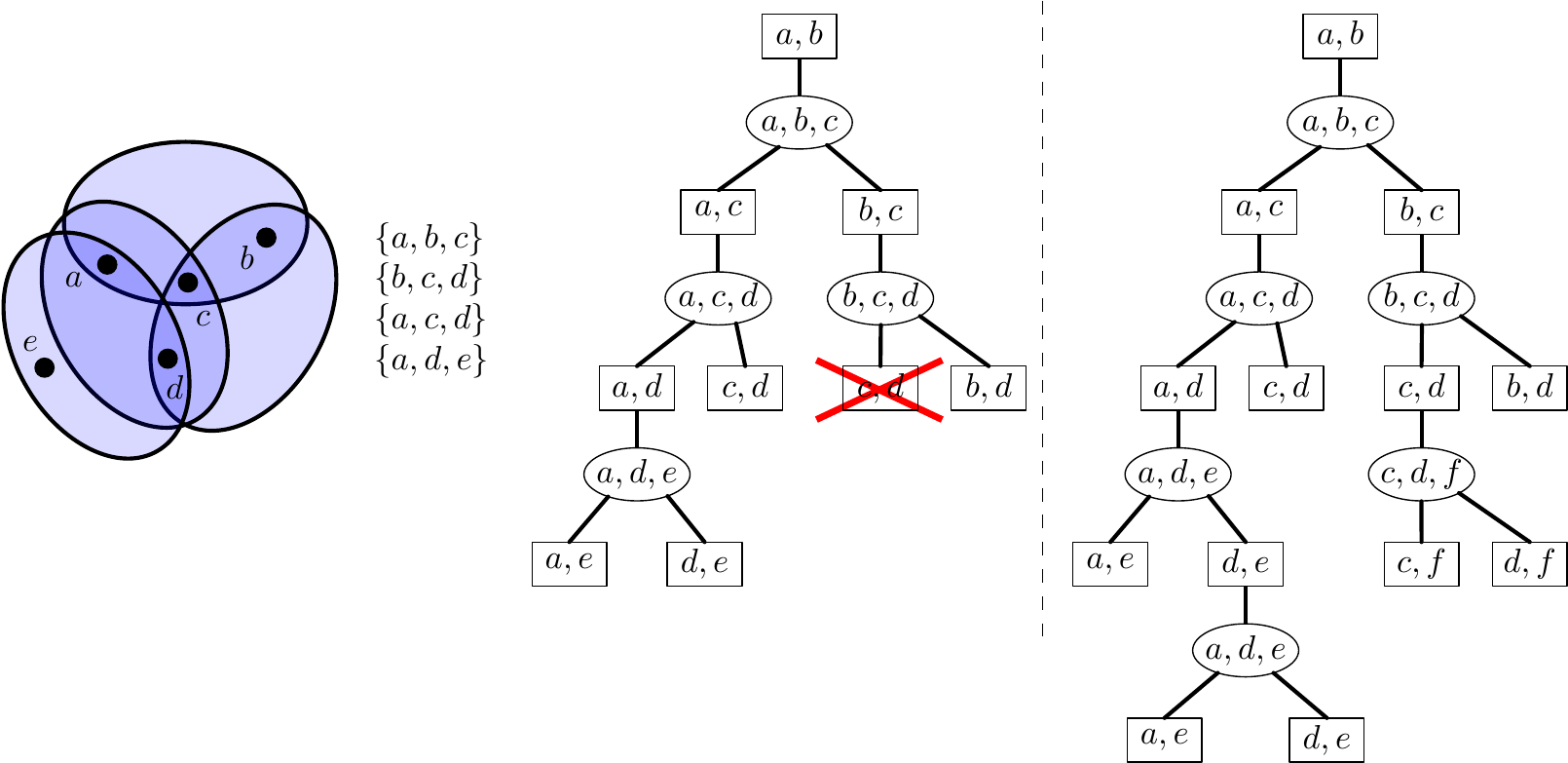}
    \caption{Examples with $k=3$, $j=2$. Left: a $j$-component with the two-type tree of its search process starting from $\{b,c\}$. Right: an instance of the branching process, giving a two-type tree that cannot come from the search process.}
  \label{fig:processes}
\end{figure}

The \tdef{component search process}\label{term:compsearchproc} is defined as follows.
We explore $j$-components in $\cH$ via a breadth-first search algorithm:
we keep track of disjoint sets of \emph{neutral}, \emph{active} and \emph{explored} $j$-sets and $k$-sets.
Furthermore, we refer to $j$-sets and $k$-sets that are either active or explored as \emph{discovered}.
Initially all $j$-sets and $k$-sets are neutral.
We use a queue to stock active
$j$-sets and $k$-sets, and at the start of the algorithm we choose an initial $j$-set $J_0$,
label it as active and add it to the queue.
When the queue is not empty, an element pops out from the queue (in a first-in first-out fashion).
If the popped element is a $j$-set $J_*$, then we consider every $k$-set $K$ containing $J_*$
in arbitrary order, and if $K$ is in $\cH$ but not yet discovered, then we call it active
and add it to the queue, and label $J^*$ as explored;
if the popped element is a $k$-set $K_*$, then we consider every $j$-set $J \subset K_*$ in arbitrary order, and if $J$ is not yet discovered, then we call it \emph{active} and add it to the queue, and label $K^*$ as explored.
We continue until the queue is empty, when we have found the $j$-component containing $J_0$.
To obtain all the $j$-components, we only need to perform the same procedure on neutral $j$-sets in $\cH$ until exhaustion. (An example of the search process is given in Figure~\ref{fig:processes}.)

To prove Lemma~\ref{lem:upper-bound}, we provide in Lemma~\ref{lem:branchcoupling} an upper coupling on the search process with the following \tdef{two-type branching process}\label{term:branchproc}. Each vertex of the branching process is either of type $j$ or of type $k$, and is labelled by a $j$-set or $k$ set of $[n]$, respectively. The branching process constructs a labelled two-type tree with label set $[n]$ in the following way. Given a vertex $u_0$ of type $j$, the branching process begins with $u_0$ (as the root). For each vertex $u$ of type $j$ with label $J$, let $\cK_J:=\{ K\in \binom{[n]}{k} \mid J\subset K \}$ be the set of possible labels of children of $u$. For each label $K \in \cK_J$, independently with probability $p$, we generate a new  vertex $v$ of type $k$ and assign the label $K$ to $v$. For each such new vertex $v$, we then attach $\binom{k}{j}-1$ many new vertices of type $j$ as children of $v$, with distinct labels from $\binom{K}{j}\setminus\{J\}$ (note that different vertices may have the same label). We denote this branching process by $\cT$\label{class:branchproc}. The size of $\cT$ is defined as the number of vertices of type $k$ that it discovers. Note that each instance of $\cT$ corresponds to a labelled two-type tree rooted at a vertex of type $j$, which we call a  \tdef{rooted labelled two-type tree}. An example of the two-type branching process is given in  Figure~\ref{fig:processes}.

\begin{lem} \label{lem:branchcoupling}
We can couple the component search process on $j$-components of $\cH$ from above with $\binom{n}{j}$ copies of $\cT$. In particular, for any given $s\in \mathbb N$, the expected number of $j$-components of $\cH$ of size at least $s$ is not larger than the expected number of instances of $\cT$ of size at least $s$.
\end{lem}

\begin{proof}
In the component search process, whenever a $j$-set $J$ becomes active, the set of $k$-sets we may query is certainly contained in $\cK_J$ (some may not be permissible since they have already been queried), and for each $k$-set $K$ discovered in this way, the $j$-sets that become active are all in $\binom{K}{j}\setminus \{J\}$ (some may not become active if they were already discovered). Thus, in $\cT$ we may have made some additional queries which are not made in the component search process, and we may have added some vertices of type $j$ whose labels correspond to $j$-sets not added in the search process. Hence, one component search process terminated when the component is fully discovered can certainly be coupled with one instance of $\cT$.

Since we need at most $\binom{n}{j}$ component search process to discover all $j$-components, we can couple with $\binom{n}{j}$ branching processes starting from vertices of type $j$ with all possible labels $J \in \binom{[n]}{j}$. More precisely, whenever we start exploring a new $j$-component from a $j$-set $J$, we upper couple this portion of the component search process by the branching process starting from a vertex of type $j$ with label $J$; using in total $\binom{n}{j}$ branching processes (although some of them may be left unused) we have an upper coupling for the search process on all the $j$-components.
\end{proof}

\section{Upper bound on $L_1$: Proof of Lemma~\ref{lem:upper-bound}} \label{sec:upper}

To prove Lemma~\ref{lem:upper-bound} we first bound the number of possible rooted two-type trees that can be constructed by the two-type branching process $\cT$. Let $\cB$\label{class:B} be the set of all possible instances of $\cT$ and thus also of all rooted labelled two-type trees. For each $s\in \mathbb N$, let $\cB_s$ be the set of elements in $\cB$ of size $s$, which is equal to the set of all rooted labelled two-type trees with $s$ vertices of type $k$. Let $B_s$ be the cardinality of $\cB_s$. In the following lemma we determine the order of $B_s$.

\begin{lem}\label{lem:branchingconfigs}
For $s\in \mathbb N$, we have
\[\binom{n}{j}\binom{n-j}{k-j}^s \frac{c_0^{s-1}s^{s-1}}{s!} \le B_s \le \binom{n}{j}\binom{n-j}{k-j}^s \frac{c_0^{s-1}s^{s-1}}{s!}e^{1/c_0}.\]
In particular, we have
\[B_s = \Theta \left(\binom{n}{j}\binom{n-j}{k-j}^s \frac{(c_0 e)^{s}}{s^{3/2}} \right).\]
\end{lem}

\begin{proof}
We first consider the class of rooted (unlabelled) two-type trees with a distinguished vertex $J$ of type $j$ as the root and at least one vertex of type $k$. Its generating function $T_J=T_J(z)$\label{genfun:TJ}, where $z$ indicates the number of vertices of type $k$, satisfies the following equation:
\begin{equation} \label{eq:T_J}
T_J(z) = \exp\left(z (1+T_J(z))^{c_0}\right) - 1.
\end{equation}

Let $W(z)$ denote the Lambert $W$-function defined by the equation $z = W(z)\exp(W(z))$. We have
\begin{equation} \label{eq:T-in-Lambert}
  T_J(z) = \exp\left( -\frac{W(-c_0 z)}{c_0} \right) - 1.
\end{equation}
By the Lagrange Inversion Theorem (see \cite[Appendix~A.6]{FlajoletSedgewickBook}), we have
\begin{equation} \label{eq:Lambert-power}
W^r(z) = \sum_{i \geq r} \frac{-r(-i)^{i-r-1}}{(i-r)!} z^i.
\end{equation}

Since $W(z)$ is analytic in a neighbourhood of $z=0$ and $W(0)=0$, for any function $F(z)$ analytic in a neighbourhood of $z=0$, the composition $F(W(z))$ is still analytic near $z=0$. Thus, using Taylor expansion we have
\begin{align*}
  T_J(z) &= \exp\left( -\frac{W(-c_0 z)}{c_0} \right) - 1 = \sum_{r \geq 1} \frac1{r!} \left( -\frac{W(-c_0 z)}{c_0} \right)^r \\
    &= \sum_{r \geq 1} \frac1{r!} \sum_{i \geq r} \frac{rc_0^{i-r}i^{i-r-1}}{(i-r)!} z^i = \sum_{i \geq 1} z^i \left( \sum_{r=1}^i \frac{c_0^{i-r} i^{i-r-1}}{(r-1)!(i-r)!} \right).
\end{align*}
The number $F_s$ of rooted unlabelled two-type trees with $s$ vertices of type $k$ is $[z^s]T_{J}(z)$,
\textit{i.e.}\ the coefficient of $z^s$ in $T_{J}(z)$. We obtain
\[
  F_s =[z^s]T_J(z) = \sum_{r=1}^s \frac{c_0^{s-r}s^{s-r-1}}{(r-1)!(s-r)!} = \frac{c_0^{s-1} s^{s-2}}{(s-1)!} \sum_{r=0}^{s-1} \frac{c_0^{-r} s^{-r} (s-1)_{(r)}}{r!}.
\]

As an upper bound, we have
\[
  F_s \leq \frac{c_0^{s-1} s^{s-1}}{s!}  \sum_{r=0}^{s-1} \frac{c_0^{-r}}{r!} 
      \leq \frac{c_0^{s-1} s^{s-1}}{s!}  e^{1/c_0}.   
\]

As a lower bound, we have
\[
  F_s \geq \frac{c_0^{s-1} s^{s-1}}{s!}.
\]

In each instance of $\cT$, we have $\binom{n}{j}$ choices for the label of the initial vertex of type $j$ and for each vertex of type $k$ discovered from a vertex of type $j$, we must choose $k-j$ new elements among $n-j$ (excluding those already in the parent vertex of type $j$). Then the labels of the next $c_0$ vertices of type $j$ are already determined. Therefore, we have the following relation between $B_s$ and $F_s$:
\[
  B_s = \binom{n}{j}\binom{n-j}{k-j}^s  \ F_s.
\]
We thus deduce the claimed bounds of $B_s$ using bounds on $F_s$. Moreover, by Stirling's approximation, we obtain the asymptotic behaviour of $B_s$:
\[
  B_s = \Theta \left(\binom{n}{j}\binom{n-j}{k-j}^s \frac{(c_0 e)^{s}}{s^{3/2}} \right). \qedhere
\]
\end{proof}

\begin{proof}[Proof of Lemma~\ref{lem:upper-bound}]
We consider $\binom{n}{j}$ independent instances of the branching process $\cT$.
Let $R_s$\label{var:Rs} be the random variable which counts the number of vertices of type $j$ present in total in the instances that have size $s$. In each such instance, the $s$ vertices of type $k$ are present with probability $p^s$. For the number of absent vertices of type $k$ (\textit{i.e.}\ which are not selected during the process) in an instance of size $s$, we observe that the $j$-set that labels the starting vertex of type $j$ is contained in $\binom{n-j}{k-j}$ many $k$-sets, and subsequently for each of the $s$ vertices of type $k$ we discover $c_0$ further vertices of type $j$, whose labels are each contained in $\left( \binom{n-j}{k-j} -1 \right)$ further $k$-sets. However, we have to consider the $s$ vertices of type $k$ that are indeed discovered. The total number of absent vertices of type $k$ is therefore
\[
  \binom{n-j}{k-j} + c_0 s \left(\binom{n-j}{k-j} -1 \right) - s = (1 + c_0 s) \binom{n-j}{k-j} - s (1+c_0).
\] 
Therefore, we have
\begin{align*}
\EX (R_s) & \le B_s p^s (1-p)^{(1+c_0 s)\binom{n-j}{k-j}-s(1+c_0)}\\
& = \Theta(1) \binom{n}{j} \left( \binom{n-j}{k-j} c_0 e p (1-p)^{c_0\binom{n-j}{k-j}}  \right)^s (1-p)^{\binom{n-j}{k-j} -s(1+c_0)} s^{-3/2}.
\end{align*}
The last equality is due to Lemma \ref{lem:branchingconfigs}. Recall that $p=(1-\eps) c_0^{-1}\binom{n-j}{k-j}^{-1}$. Using $(1-p) \le e^{-p}$ we have
\begin{align*}
  \EX (R_s)  &\le \Theta(1) \binom{n}{j} \left( (1-\eps) \cdot e \cdot e^{-(1-\eps)} \right)^s s^{-3/2} \exp\left((1-\eps)(1+c_0^{-1})s\binom{n-j}{k-j}^{-1} -\frac{1-\eps}{c_0}\right) \\
  &\leq \Theta(1) \binom{n}{j} \exp \left( s ( \log (1-\eps)  + \eps ) \right) s^{-3/2} \exp\left(2s\binom{n-j}{k-j}^{-1}\right).
\end{align*}
Since $\delta = -\eps - \log(1-\eps)$, we have
\begin{align*}
\EX (R_s) & \le  \Theta(1)  \binom{n}{j} \exp \left( -s \delta  \right) s^{-3/2} \exp\left(2s\binom{n-j}{k-j}^{-1}\right).
\end{align*}

Now, let $D_s$\label{var:Ds} be the random variable which counts the number of components of $\cH$ of size $s$. By Lemma \ref{lem:branchcoupling}, $\EX(D_s)$ is bounded above by the expected number of instances of the two-type branching process $\cT$ of size $s$. In each such instance every vertex of type $k$ is connected to exactly $c_0 + 1$ vertices of type $j$, therefore their expected number is equal to $\EX(R_s)(c_0+1)^{-1} s^{-1}$, where $R_s$ is the number of vertices of type $j$ present in total in all the instances. Denoting by $D_{\geq s}= \sum_{t\geq s} D_t$ the random variable counting the number of components of size at least $s$, we have
\[ \EX(D_{\geq s}) \leq \sum_{t\geq s} \EX(R_t)(c_0+1)^{-1} t^{-1}.\]
Let $\hat{s}= \delta^{-1}(\log \lambda - \frac{5}{2} \log \log \lambda+ K(n))$ and $\delta^- = \delta - 2 \binom{n-j}{k-j}^{-1}$. Recalling that $\delta = \eps^2/2 + O(\eps^3)$ and $\eps^2 n^{k-j} (\log n)^{-1} \to \infty$, it holds that $\delta = \omega(n^{j-k} \log n)$, which means $\delta^- = \delta - 2\binom{n-j}{k-j}^{-1} = (1-o(1)) \delta$. Thus, we have $\delta^- > 0 $ for $n$ large enough and we obtain
\begin{align*}
  \EX(D_{\geq \hat{s}}) &\leq \sum_{t\geq \hat{s}} \EX(R_t) (c_0+1)^{-1} t^{-1} \le \Theta(1) \binom{n}{j}(\hat{s})^{-5/2} \sum_{t\geq \hat{s}} (e^{-\delta^-})^t\\
& \leq \Theta(1) \binom{n}{j}(\hat{s})^{-5/2} \frac{e^{-\delta^- \hat{s}}}{1 - e^{-\delta^-}} \leq \Theta(1) \binom{n}{j}(\hat{s})^{-5/2} \frac{e^{-\delta^- \hat{s}}}{\delta^-}.
\end{align*}

Without loss of generality, we may assume that $K(n)=o(\log \lambda)$. Therefore, since $\frac{\eps^2 n^{k-j}}{\log n} \to \infty$, we have
\begin{align*}
\hat{s} \binom{n-j}{k-j}^{-1} &= (1+o(1) ) \delta^{-1} (\log \lambda) \binom{n-j}{k-j}^{-1} = O(1) \frac{\log (\eps^3 n^j)}{\eps^2 n^{k-j}} \\
&\leq O(1) (\log n) \eps^{-2} n^{j-k} = o(1).
\end{align*}
This leads to
\[
e^{-\delta^- \hat{s}} = e^{-\delta \hat{s}} \exp\left(2\hat{s}\binom{n-j}{k-j}^{-1}\right) =  (1+o(1)) e^{-\delta \hat{s}}.
\]

Since $\eps^4 n^j \to \infty$, we have $\lambda = \eps^3 \binom{n}{j} \rightarrow \infty$. We thus obtain
\begin{align*}
  \EX(D_{\geq \hat{s}}) &\leq \Theta(1) \binom{n}{j}(\hat{s})^{-5/2} \frac{(1+o(1)) e^{-\delta \hat{s}}}{(1+o(1))\delta} = \Theta(1) \binom{n}{j}(\hat{s})^{-5/2} \frac{e^{-\delta \hat{s}}}{\delta} \\
  &\le \Theta(1) \exp \left( \log \binom{n}{j} 
   - \frac{5}{2} \log \left((1+o(1))\frac{\log \lambda}{\delta} \right) - \log \lambda + \frac{5}{2} \log \log \lambda -K(n) - \log \delta \right) \\
& = \Theta(1) \exp \left( \log \binom{n}{j} 
   + \frac{5}{2} \log \delta  - \log \left( \eps^3 \binom{n}{j} \right)  -K(n) - \log \delta \right) \\
& = \Theta(1) \exp \left( \frac{3}{2} \left(\log (\eps^2)  - \log 2 +  \log \left(1+ O(\eps)\right) \right) - \log ( \eps^3 ) -K(n)  \right) \\
&  = O \left( \exp \left(- K(n) \right) \right).
\end{align*}
Since $K(n) \to \infty$, by Markov's inequality, whp we have $D_{\geq \hat{s}}=0$, meaning that there is no $j$-component of size larger than $\hat{s} = \delta^{-1}(\log \lambda - \frac{5}{2} \log \log \lambda+ K(n))$.
\end{proof}

\section{Lower bound on $L_1$: Proof of Lemma~\ref{lem:lower-bound-second-moment}} \label{sec:lower}
In this section we will prove that $\cH$ contains a hypertree component larger than a certain size,
which provides a lower bound on $L_1$ (Lemma~\ref{lem:lower-bound-second-moment}).
To this end, we study the only obstacle for a $j$-component to be a hypertree,
which are the \emph{wheels}, a hypergraph analogue of cycles.
We first give an upper bound on the number of possible wheels of length $\ell\in \mathbb{N}$ (Lemma~\ref{lem:wheelconfigs}).
This then implies that almost all possible (not too large) instances of the two-type branching process
correspond to hypertrees (Lemma~\ref{lem:bound-wheels}).
Lemma~\ref{lem:wheelconfigs} also implies that whp
there are no large non-hypertree components in $\cH$ (Lemma~\ref{lem:non-hypertree-bound}).
The proofs of these auxiliary lemmas will be delayed until Section~\ref{sec:wheels}.

Before giving these proofs, in this section we use the auxiliary lemmas to determine the asymptotic first and second moments
of the number of large components in $\cH$. Since the second moment is approximately the square of the first,
Chebyshev's inequality implies that whp there are many such large components, which by
Lemma~\ref{lem:non-hypertree-bound} are hypertrees, as required.

Recall that a hypertree component (\textit{i.e.}\ $j$-component that contains as many $j$-sets as possible
given its size) corresponds to a labelled two-type tree with no repeated labels.
An important structure that may appear in $\cH$ is the so-called \tdef{wheel}.
A \tdef{wheel}\label{term:wheel} of length $\ell \geq 2$ is a pair of sequences,
one of $\ell$ distinct hyperedges $K_1, K_2, \ldots, K_\ell, K_{\ell+1}=K_1$
and the other of $\ell$ distinct $j$-sets $J_0, J_1, \ldots, J_{\ell-1}, J_\ell=J_0$
such that $J_i \subset K_i \cap K_{i+1}$ for all $1 \leq i \leq \ell$ (see Figure~\ref{fig:wheel}).
Two wheels are considered identical if they only differ by a cyclic rotation or order reversion of the elements of the sequences. Given a wheel, it lies in a single $j$-connected component, and the presence of a wheel is the only obstacle for a component to be a hypertree. The reason is that a component ceases to be a hypertree if and only if we encounter the same $j$-set or $k$-set at least twice in the component search process, which makes a wheel. We have the following enumeration result on wheels.

\begin{figure}
\centering
\includegraphics[page=3,scale=1]{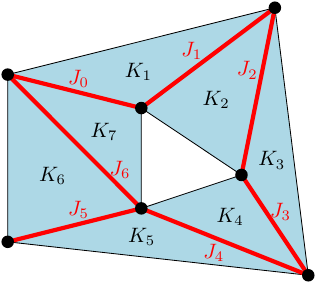}
\caption{An example of a wheel in the case $k=3$ and $j=2$.}
\label{fig:wheel}
\end{figure}

\begin{lem} \label{lem:wheelconfigs}
  Let $w_\ell = w_\ell(n)$\label{par:wl} be the number of possible wheels of length $\ell \geq 2$, with vertices chosen from $[n]$. We have
  \[
  w_\ell \leq \frac{c_w n^{k-j}}{\pg^{\ell-1}\ell}, \;\; \mathrm{where} \;\; c_w = \frac{(k-j)^j}{j!(k-j)!} \prod_{m=1}^{j-1} \left( 1 - c_0^{-1} \left(\binom{k-m}{j-m}-1\right) \right)^{-1}.
  \]
\end{lem}
We note that $c_w$ does not depend on $n$ or $\ell$. The detailed proof is postponed until Section~\ref{sec:wheels:wheelconfigs}.

Recall that we denote by $\cB$ the set of possible instances of $\cT$ and by $\cB_s$ the elements in $\cB$ with $s$ vertices of type $k$ (bounds on $B_s=|\cB_s|$ were given in Lemma~\ref{lem:branchingconfigs}). We now consider the subset $\cB^-$\label{class:Bmin} of $\cB$ formed by rooted labelled two-type trees in which all labels are distinct, \textit{i.e.}\ that correspond to hypertrees, and we denote by $\cB^-_s$ the set of elements in $\cB^-$ with $s$ vertices of type $k$.

\begin{lem} \label{lem:bound-wheels}
  For $s \geq 2304$, we have
  \[
  B_s^- := |\cB_s^-| = (1 - O(s n^{j-k}) - O(s^2 n^{-j}))B_s.
  \]
  In particular, if $s\to \infty$ and also $s n^{j-k} \to 0$ and $s^2 n^{-j} \to 0$, we have $B_s^- = (1-o(1))B_s$.
\end{lem}
Lemma~\ref{lem:bound-wheels} will be proved in Section~\ref{sec:wheels:bound-wheels}.

Let $C_s$\label{var:Cs} be the number of $j$-sets in hypertree components of size $s$ in $\cH$. It is clear that $C_s$ is a lower bound for the number of $j$-sets in components of size $s$. 

\begin{lem} \label{lem:lower-bound-j}
  For $s \geq 2304$ such that $s n^{j-k} \to 0$ and $s^2 n^{-j} \to 0$, we have
  \[
  \EX(C_{s}) \geq \Theta(1) \exp\left( \log \binom{n}{j} - s\delta - \frac{3}{2} \log s \right).
  \]
\end{lem}
\begin{proof}
  To give a bound on $\EX(C_{s})$, we bound the probability of each element of $\cB_s^-$ occurring in the hypergraph $\cH$. Given a hypertree of size $s$, the probability that it occurs as a component in $\cH$ is at least $p^s (1-p)^{(1+sc_0)\binom{n-j}{k-j}}$, since each $j$-set implies the absence of at most $\binom{n-j}{k-j}$ hyperedges, and since the number of $j$-sets in a hypertree component of size $s$ is exactly $sc_0+1$. By Lemma~\ref{lem:bound-wheels}, we have $B_s^-=(1-o(1))B_s$. Using Lemma~\ref{lem:branchingconfigs} and the fact that $1-p \leq e^{-p}$, we have
\begin{align*}
  \EX(C_{s}) &\geq B_s^- p^s (1-p)^{(1+sc_0)\binom{n-j}{k-j}} \\
  &\geq (1-o(1)) B_s p^s (1-p)^{sc_0 \binom{n-j}{k-j} + \binom{n-j}{k-j}}  \\
  &= \Theta(1) \binom{n}{j} \binom{n-j}{k-j}^s \frac{(c_0 e)^s}{s^{3/2}} \cdot (1-\eps)^s c_0^{-s} \binom{n-j}{k-j}^{-s} e^{-(1-\eps)s} e^{-(1-\eps)c_0^{-1}} \\
  &=\Theta(1) \binom{n}{j} \frac{e^s}{s^{3/2}} (1-\eps)^s e^{-(1-\eps)s} \\
  &=\Theta(1) \exp\left(\log\binom{n}{j} + s -  \frac{3}{2}\log s + s  \log(1-\eps)  - s (1-\eps)  \right)\\
  &=\Theta(1) \exp\left(\log\binom{n}{j} - s \left( - \eps - \log(1-\eps) \right) -  \frac{3}{2}\log s \right).
\end{align*}
We conclude the proof by recalling that  $\delta = -\eps -\log(1-\eps)$.
\end{proof}

We also consider components in $\cH$ that are not hypertree components, and therefore each of them must contain a wheel. We recall the value $\delta = -\eps -\log(1-\eps) = \eps^2/2 + O(\eps^3)$. 
We have the following lemma, which will be proved in Section~\ref{sec:wheels:non-hypertree}.

\begin{lem} \label{lem:non-hypertree-bound}
 Let $s^\circ = s^\circ(n)$ satisfy $s^\circ \delta \to \infty$. Then whp there is no non-hypertree component in $\cH$ of size larger than $s^\circ$. 
\end{lem}

In other words, whp any component containing a wheel has size at most $s^\circ$. We can now combine the results of this section to prove Lemma~\ref{lem:lower-bound-second-moment}, using a second moment argument.

\begin{proof}[Proof of Lemma~\ref{lem:lower-bound-second-moment}]
  We set $s^* = \delta^{-1} \log \lambda$ and $s_*= \delta^{-1}(\log \lambda - \frac{5}{2} \log \log \lambda - K(n))$ where $K(n)\to\infty$. We assume here $K(n) = o(\log \lambda)$ without loss of generality. Furthermore, we set $s_0 = s_* + \delta^{-1} K(n) / 2 = \delta^{-1}(\log \lambda - \frac{5}{2} \log \log \lambda - K(n)/2)$. Firstly, we know from Lemma~\ref{lem:upper-bound} that whp there is no component larger than $s^*$. Let $S_+$\label{var:S+} denote the number of $j$-sets in components of size between $s_*$ and $s^*$. Since the range of $\eps$ implies that $s_*$ and $s_0$ satisfy the conditions in Lemma~\ref{lem:lower-bound-j}, we have the following lower bound by counting only hypertree components:
  \begin{align} \label{eq:expectS+}
    \EX(S_+) &\geq \sum_{s_* \leq s \leq s_0} \EX(C_{s}) \geq (s_0 - s_*) \Theta(1) \exp\left( \log \binom{n}{j} - s_0 \delta - \frac{3}{2} \log s_0 \right) \nonumber \\ 
    & \ge \Theta(1)  \frac{K(n)}{2 \delta}  \exp\left( \log \binom{n}{j} - \log \lambda + \frac{5}{2} \log\log \lambda + \frac{K(n)}{2} - \frac{3}{2} \log  \left( \delta^{-1} \log \lambda \right) \right) \nonumber \\ 
    & \ge \Theta(1)  \frac{K(n)}{2 \delta}  \exp\left( - \log \eps^3 + \frac{5}{2} \log\log \lambda + \frac{K(n)}{2} + \frac{3}{2} \log \frac{\eps^2}{2} + \log(1+O(\eps)) - \frac{3}{2} \log  \log \lambda \right) \nonumber \\ 
    &\geq \Theta(1) \frac{K(n)}{2 \delta} \exp(K(n)/2)\log \lambda = \Theta(1) K(n) \exp(K(n)/2) s^*= \omega(s^*). 
  \end{align}
We now show that $\EX(S_+^2)$ is approximately $\EX(S_+)^2$. Let $q$ be the probability that a given $j$-set is in a component of size between $s_*$ and $s^*$, so $\EX(S_+) = q \binom{n}{j}$. For two $j$-sets $J_1, J_2$ (not necessarily distinct), let $\kappa_1, \kappa_2$ be the components in which they lie, respectively. Let $s_1, s_2$ be the sizes of $\kappa_1$ and $\kappa_2$ respectively. We have
  \begin{align*}
    \EX(S_+^2) &= \sum_{J_1, J_2} \Pr( s_* \leq s_1 \leq s^*,\; s_* \leq s_2 \leq s^*) \\
    &\leq \sum_{J_1} \Pr(s_* \leq s_1 \leq s^*) \sum_{J_2} \Pr(s_2 \geq s_* \mid s_* \leq s_1 \leq s^*).
  \end{align*}
  Given that $\kappa_1$ is of size between $s_*$ and $s^*$, we want to bound the probability that $\kappa_2$ is of size at least $s_*$. Given the $j$-set $J_2$, if $J_2 \in \kappa_1$, then $\kappa_2 = \kappa_1$ is of size at least $s_*$; otherwise, we start a modified search process in the hypergraph, where we ignore an existing hyperedge whenever it contains a $j$-set in $\kappa_1$. This modified search process can be upper coupled by the unmodified search process $\cT$, in which all $j$-sets and $k$-sets are still available. Therefore, we have
  \begin{align*}
    \EX(S_+^2) &\leq \sum_{J_1} \Pr(s_* \leq s_1 \leq s^*) \left( s_1 + \left( \binom{n}{j} - s_1 \right) q \right) \\
    &\leq \EX(S_+) \left( s^* + (1+o(1)) q \binom{n}{j} \right) = \EX(S_+)^2 \left( 1 + o(1) + \frac{s^*}{\EX(S_+)} \right) \stackrel{\eqref{eq:expectS+}}{=} \EX(S_+)^2 (1+o(1)),
  \end{align*}
By Lemma~\ref{lem:upper-bound} and Chebyshev's inequality, we have
\begin{align*}
\Pr({\cH} \ \textrm{contains no component of size at least } s_*) & \leq \Pr(S_+=0) +o(1) \\ & \leq \frac{\EX(S_+^2) - \EX(S_+)^2}{\EX(S_+)^2} + o(1) = o(1).
\end{align*}

  Similarly, for any fixed constant $m$ the probability that $\cH$ contains at most $m$ components of size at least $s_*$ is bounded by $\Pr(S_+ \leq ms^*) + o(1)$. Again by Chebyshev's inequality, we have
  \[
    \Pr(S_+ \leq ms^*) + o(1) \leq \frac{\EX(S_+^2) - \EX(S_+)^2}{(\EX(S_+) - ms^*)^2} + o(1) = o(1).
  \]
  The latter equality is due to the fact that $\EX(S_+) \stackrel{\eqref{eq:expectS+}}{=} \omega(s^*)$.

  By Lemma~\ref{lem:non-hypertree-bound}, since $s_* \delta = \Theta(\log \lambda) \to \infty$, we know that whp all components of size at least $s_*$ are hypertree components, including the $m$ largest components.
\end{proof}

\section{Wheels: Proofs of auxiliary results} \label{sec:wheels}

\subsection{Proof of Lemma~\ref{lem:wheelconfigs}} \label{sec:wheels:wheelconfigs}

  To construct a wheel with $\ell$ distinct hyperedges $K_1, K_2, \ldots, K_\ell$ and $\ell$ distinct $j$-sets $J_0,J_1,\ldots,J_\ell$, we first pick a $j$-set $J_0$  and choose the other $k-j$ vertices for $K_1$. Subsequently, if we have constructed $K_i$, we pick a $j$-set $J_i$ in $K_i$ that is different from any $J_{i'}$ picked out before, and choose another $k-j$ vertices to construct $K_{i+1}$. To make sure that $K_\ell$ constructed in the end includes the initial $j$-set $J_0$, we keep track of vertices in $J_0$. For each vertex $v$ in $J_0$, we say that $J_i$ \tdef{freezes} $v$ if $v \in J_{i'}$ for every $i'$ such that $i \leq i' \leq \ell$. If $J_i$ freezes $a_i$ many vertices, then $J_i$ must be chosen to contain all these $a_i$ vertices, and if $J_{i+1}$ freezes $a_{i+1}$ many vertices, then a further $a_{i+1}-a_{i}$ many vertices must be chosen from $J_0$ when selecting the $k-j$ vertices to construct $K_{i+1}$ from $J_i$. Note that every wheel can be obtained in our construction, which gives an over-counting, since for fixed $a_i$ we could inadvertently choose vertices such that some are frozen earlier than necessary.

  Let $\tau_0, \tau_1, \tau_2, \ldots, \tau_{j-1}$ be integers such that $\tau_d$ is the number of hyperedges in the constructed wheel that freeze $d$ vertices in $J_0$, \textit{i.e.}\ the number of $a_i$'s of value $d$. Since the $a_i$'s are non-decreasing by definition, we can deduce the $a_i$'s from the $\tau_i$'s and vice versa. We now consider the number of choices in each step. There are $\binom{n}{j}$ choices for $J_0$ and $\binom{n-j}{k-j}$ choices for the remaining vertices of $K_1$. Subsequently for each $J_i$ that freezes $a_i$ many vertices, there are $\binom{k-a_i}{j-a_i}-1$ choices. Now, to obtain from $J_i$ a $k$-set $K_{i+1}$ that can contain a $J_{i+1}$ wich freezes $a_{i+1}$ many vertices, there are $\binom{n-j-a_{i+1}+a_i}{k-j-a_{i+1}+a_i}$ choices.

  By definition, it is clear that only $J_\ell$ freezes all the $j$ vertices in $J_0$, thus $a_\ell=j$ and $a_i \leq j-1$ for all $i < \ell$. 
The number of constructions $w_\ell^\star(a_1, a_2, \ldots, a_\ell)$\label{par:wlstar} with $a_1, \ldots, a_\ell$  many vertices, frozen by $J_1, \ldots, J_\ell$ respectively, is bounded by (with $\tau_i$'s computed from the $a_i$'s)
  \begin{align*}
    w_\ell^\star(&a_1, a_2, \ldots, a_\ell) \leq \binom{n}{j} \binom{n-j}{k-j} \left[ \prod_{i=1}^{\ell-1} \binom{n-j-a_{i+1}+a_i}{k-j-a_{i+1}+a_i} \right] \prod_{m=0}^{j-1} \left( \binom{k-m}{j-m}-1 \right)^{\tau_m} \\
                                             &\leq \binom{n}{j} \binom{n-j}{k-j} \left[ c_0 \binom{n-j}{k-j} \right]^{\ell-1} \left[ \prod_{i=1}^{\ell-1} \left( \frac{k-j}{n-j-a_{i+1}+a_{i}} \right)^{a_{i+1}-a_i} \right] \prod_{m=1}^{j-1} \left( \frac{\binom{k-m}{j-m}-1}{c_0} \right)^{\tau_m} \\
                                             &\leq \frac{n^k}{j!(k-j)!\pg^{\ell-1}} \left( \frac{k-j}{n-2j} \right)^j \prod_{m=1}^{j-1} \left( \frac{\binom{k-m}{j-m}-1}{c_0} \right)^{\tau_m}.
  \end{align*}

  Noting that wheels are considered identical up to rotation and reversed order, and that $\binom{k-m}{j-m}-1 < c_0$ for $1 \leq m \leq j-1$, by summing over all possible $\tau_i$'s (and thus also $a_i$'s) we have
  \begin{align*}
    w_\ell &\leq \frac1{2\ell} \sum_{0 \leq a_1 \leq a_2 \leq \cdots \leq a_\ell=j} w_\ell^\star(a_1, a_2, \ldots, a_\ell) \\
    &\leq \frac{n^k}{2\ell \cdot j!(k-j)!\pg^{\ell-1}} \left( \frac{k-j}{n-2j} \right)^{j} \prod_{m=1}^{j-1} \sum_{\tau_m \geq 0} \left( \frac{\binom{k-m}{j-m}-1}{c_0} \right)^{\tau_m} \\
    &\leq \frac{n^{k-j}\left(1+\frac{3j^2}{n}\right)(k-j)^j}{2\ell \cdot j!(k-j)!\pg^{\ell-1} } \prod_{m=1}^{j-1} \left( 1 - \frac{\binom{k-m}{j-m}-1}{c_0} \right)^{-1} \leq \frac{c_w n^{k-j}}{\pg^{\ell-1}\ell},
  \end{align*}
  as required.\qed
  
\subsection{Proof of Lemma~\ref{lem:bound-wheels}} \label{sec:wheels:bound-wheels}

We first introduce the notion of dominance for comparing generating functions. Recall that given a generating function $F(z)$, we denote by $[z^n]F(z)$ the coefficient of $z^n$ in $F(z)$. For two generating functions $F(z)$ and $G(z)$, we say that $F(z)$ is \tdef{dominated} by $G(z)$ (denoted by $F(z) \preceq G(z)$) if for all $n \in \mathbb{N}$ we have $[z^n]F(z) \leq [z^n]G(z)$. For generating functions whose coefficients are non-negative integers, the dominance relation is clearly preserved under by addition, multiplication, differentiation by $z$ and composition.

We denote by $\cB^\circ$\label{class:Bcirc} the set $\cB \setminus \cB^-$, \textit{i.e.}\ the set of all rooted labelled two-type trees that do not correspond to hypertrees, and $\cB^\circ_s$ the corresponding set of elements with $s$ vertices of type $k$. We first look at the structure of elements in $\cB^\circ_s$. We define a \tdef{two-type unicycle}\label{term:twotypunic} as a labelled two-type graph obtained by attaching rooted labelled two-type trees or nothing to every vertex of type $j$ in a labelled two-type graph generated by hyperedges in a wheel and the $j$-sets they contain. Note that vertices may share the same label.

\begin{prop} \label{prop:decomp-non-hypertree}
  There is an injection from $\cB^\circ_s$ to the disjoint union of the three sets $Q_s^1, Q_s^2, Q_s^3$, where
  \begin{enumerate}
  \item $Q_s^1$ is the set of tuples $(T^\bullet, (T_1, T_2, \ldots, T_{c_0}))$ with $s-1$ vertices of type $k$ in total, where $T^\bullet \in \cB$ contains a marked non-root vertex $w$ of type $j$, and the $T_i$'s are either empty or in $\cB$, but the labels of their roots are fixed to the $c_0$ many $j$-subsets of the label of the parent of $w$, excluding that of $w$;
  \item $Q_s^2$ is the set of tuples $(C^{\bullet, k}, (T_1, T_2, \ldots, T_{c_0}))$ with $s-1$ vertices of type $k$ in total, where $C^{\bullet,k}$ is a two-type unicycle with a marked vertex $u$ of type $k$ in the cycle and a marked vertex of type $j$, and $T_i$'s are either empty or in $\cB$, but the labels of their roots are fixed to the $c_0$ many $j$-subsets of the label of the marked vertex $u$, excluding that of the parent of $u$ in a breadth-first search starting from the marked vertex of type $j$ using the order provided by labels;
  \item $Q_s^3$ is the set of tuples $(C^{\bullet, j}, T_0)$ with $s$ vertices of type $k$ in total, where $C^{\bullet, j}$ is a two-type unicycle with a marked vertex $u$ of type $j$ in the cycle and another marked vertex of type $j$, and $T_0$ is  either empty or in $\cB$, with the label of its root the same as that of $u$.
  \end{enumerate}
\end{prop}

\begin{proof}

 Let $C$ be an element of $\cB^\circ_s$. Since $C$ does not correspond to a hypertree, we know that $C$ has at least a pair of vertices with the same label. We perform a breadth-first search (with the order provided by labels), until we meet a vertex $v$ with the same label as some other vertex $u$ that comes before. There are three possibilities:
  \begin{enumerate}[(1)]
  \item $u, v$ are of type $k$, and $u$ is the grandparent of $v$;
  \item $u,v$ are of type $k$, and are at distance at least $4$;
  \item $u,v$ are of type $j$.
  \end{enumerate}

In case (1), we take out $v$ and its $c_0$ sub-trees $T_1, \ldots, T_{c_0}$ rooted at $j$-sets, and we mark the vertex of type $j$ between $u$ and $v$. We thus get a rooted labelled two-type tree $T^\bullet$ with one marked vertex of type $j$ along with the $T_i$'s. Since $C \in \cB^\circ_s$ but $v$ has been removed, we have $s-1$ vertices of type $k$. Therefore, the tuple $(T^\bullet, (T_1, T_2, \ldots, T_{c_0}))$ is in $Q_s^1$.

  In case (2), we take out the $c_0$ sub-trees $T_1, \ldots, T_{c_0}$ of $v$, we merge $v$ and $u$ and we mark $u$. The merging gives a cycle formed by the path from $u$ and $v$ to their lowest common ancestor in the tree, which corresponds to a wheel because no other vertices have the same label in the explored part. We also have two-type trees attached to its type $j$ vertices. We also regard the original root as a marked vertex of type $j$. This is the two-type unicycle $C^{\bullet, k}$. Since $C \in \cB^\circ_s$ and $v$ has been merged with $u$, we have $s-1$ vertices of type $k$. Labels of the root of the $T_i$'s (when not empty) are fixed to be subsets of size $j$ of the label of $v$, excluding the label of its (original) parent. Therefore, the tuple $(C^{\bullet, k}, (T_1, T_2, \ldots, T_{c_0}))$ is in $Q_s^2$.

  In case (3), by construction, the parents of $u$ and $v$ are different. We now separate off the sub-tree $T_0$ rooted at $v$, and merge $u$ and $v$. Similarly to the second case, we have a wheel with a marked vertex of type $j$ and two-type trees attached to its $j$-sets, and also the root as a marked vertex of type $j$ that can be anywhere. This is the two-type unicycle $C^{\bullet, j}$. No type $k$ vertex has been removed or merged, so there are $s$ in total. The label of the root of $T_0$ is fixed to be the label of $u$. Therefore, the tuple $(C^{\bullet, j}, T_0)$ is in $Q_s^3$.

  To show that this indeed defines an injection, we only need to observe that for each construction, the reverse direction has only one possibility. We can identify which of the three cases we are in from the first element of the resulting tuple.
  In case (1), given $(T^\bullet, (T_1, T_2, \ldots, T_{c_0})) \in Q_s^1$, we add a new child to the marked vertex $w$ of type $j$ with the same label as the parent of $w$, and attach all $T_i$'s as its sub-trees. The root of the new tree is the root of $T^\bullet$.
  In case (2), given $(C^{\bullet, k}, (T_1, T_2, \ldots, T_{c_0})) \in Q_s^2$, the marked vertex $u$ of type $k$
  is adjacent to two vertices $v_1,v_2$ of type $j$ on the cycle (where $v_1$ has the smaller of the two labels).
  We replace $u$ by $u_1$ and $u_2$, and connect $u_1$ to $v_1$ and $u_2$ to $v_2$ (thus breaking up the cycle
  into a path). Additionally, all other vertices which were adjacent to $u$ in $C^{\bullet, k}$ are connected
  to $u_1$. Finally the roots of the $T_i$'s are connected to $u_2$. We thus obtain, a tree which we root
  at the marked vertex of type $j$.
  The construction for case (3) is similar to that of case (2).
\end{proof}

Before computing $B_s^- = |\cB_s^-|$ (Lemma~\ref{lem:bound-wheels}), we give a technical lemma in the spirit of Laplace's method.

\begin{lem} \label{lem:laplace-app}
  Given any fixed integer $a \geq 1$, for $s \geq (16a)^2$, we have
  \[
  \sum_{i = 1}^s \frac{i^a s_{(i)}}{s^{i}} \leq 5 (2a)^{a/2} s^{(a+1)/2}.
  \]
\end{lem}
\begin{proof}
  We first observe that
  \[
  \frac{i^a s_{(i)}}{s^i} = \exp\left( a \log i + \sum_{b=0}^{i-1} \log \left( 1 - \frac{b}{s} \right) \right) \leq \exp\left( a \log i - \frac{i^2}{4s} \right).
  \]
  Let $S(i) = a \log i - i^2 / (4s)$. The maximum of $S(i)$ (viewed as a function on $\mathbb{R}$ rather than $\mathbb{N}$) occurs at $i_{opt} = (2as)^{1/2}$, with value $S(i_{opt}) = \frac{a}{2}(\log(2as)-1)$. Since $S''(i) = -a/i^2 - 1/(2s) < 0$ for $1 \leq i \leq s$, we know that $S(i)$ is concave in this range. We also observe that, for $\alpha > 0$, we have
\begin{equation} \label{eq:Salphaiopt}
S(\alpha i_{opt}) = a \log i_{opt} + a \log \alpha - a \frac{(\alpha i_{opt})^2}{2i_{opt}^2} = S(i_{opt}) + a \left(\log \alpha + \frac1{2} - \frac{\alpha^2}{2} \right).
\end{equation}

  Let $i^* = \lceil (16as \log s)^{1/2} \rceil$. It is clear that $i^* \geq i_{opt} + 1$ for $s \geq (16a)^2$ with $a \geq 1$. Since $S(i)$ is concave and $i^* \geq i_{opt}+1$, we have 
  \begin{align*}
    S(i^*) \leq S((16as\log s)^{1/2}) & = a \log \left((16as \log s)^{1/2}\right) - \frac{16as \log s}{4s} \\
    &= \frac{a}{2} \left(\log s + \log(16a) + \log\log(s)\right) - 4a\log s \\
    &\leq \frac{a}{2}(\log\log s - 7\log s) + \frac{a}{2}\log(16a),
  \end{align*}
  which is decreasing in $s$. In the case $s=(16a)^2$, we have $S(i^*) = a \log \log (16a) - \frac{13a}{2} \log(16a)$, which is clearly negative for $a \geq 1$. Therefore, $S(i^*) < 0$ holds for all $s \geq (16a)^2$. Therefore, by the concavity of $S(i)$ and the fact that $i_{opt} < i^*$ for our range of $s$, we have
  \begin{align*}
    \sum_{i=1}^s \frac{i^a s_{(i)}}{s^{i}} &\leq \sum_{i=1}^{i^*} \exp\left(a \log i - \frac{i^2}{4s} \right) + s \leq s + \exp(S(i_{opt})) + \int_{1}^{i^*} \exp(S(x)) dx \\
                                         &\stackrel{\eqref{eq:Salphaiopt}}{\leq} s + \exp(S(i_{opt})) + i_{opt} \exp(S(i_{opt})) e^{a/2} \int_{0}^{+\infty} \exp(a(\log \alpha - \alpha^2/2)) d\alpha \\
                                         &\leq s + e^{-a/2}(2as)^{a/2} + (2as)^{(a+1)/2} \int_{0}^{+\infty} \exp(-a\alpha^2/4)d\alpha \\
                                         &\leq s + (2as)^{a/2} + 2 \sqrt{\pi} (2a)^{a/2} s^{(a+1)/2} \\
                                         &\leq (\sqrt{1/2}+1+2\sqrt{\pi}) (2a)^{a/2} s^{(a+1)/2} \leq 5 (2a)^{a/2}s^{(a+1)/2}.
  \end{align*}
  In the third line we used $\log \alpha < \alpha^2 /4$, which holds for all $\alpha > 0$. This is because $\log \alpha - \alpha^2 /4$ takes its maximum at $\alpha=\sqrt{2}$, where it has a negative value.
\end{proof}

  We now consider the generating function of two-type graphs corresponding to wheels
\begin{equation} \label{eq:w}
  w(z) = \sum_{\ell \geq 2} w_\ell z^\ell,
\end{equation}  
 and that of the set $\cB_s$ and $\cB_s^\circ$, denoted by $B(z)$ and $B^\circ(z)$ respectively. By Proposition~\ref{prop:decomp-non-hypertree}, we can partition $\cB^\circ_s$ into three disjoint subsets $\cB^{\circ,1}_s$, $\cB^{\circ,2}_s$ and $\cB^{\circ,3}_s$ (\textit{i.e.}\ the preimages of $Q_s^1$, $Q_s^2$ and $Q_s^3$ under the injection) with generating functions $B^{\circ,1}(z)$, $B^{\circ,2}(z)$, $B^{\circ,3}(z)$ respectively. We thus have $B^\circ(z) = B^{\circ,1}(z) + B^{\circ,2}(z) + B^{\circ,3}(z)$, where $[z^s]B^{\circ,i}(z) = |\cB^{\circ,i}_s|$ for $i=1,2,3$. 

  From Lemma~\ref{lem:wheelconfigs} and the fact that a wheel consists of at least 2 hyperedges, we have
  \begin{equation} \label{eq:boundw}
    w(z) \preceq c_w n^{k-j} \pg \left(\log \frac1{1-\pg^{-1}z} - \pg^{-1}z \right).
  \end{equation}
  The generating function $w^{\bullet}(z)$\label{genfun:wbullet} of two-type graphs corresponding to wheels with a marked $k$-set is therefore
  \begin{equation} \label{eq:boundwbullet}
  w^\bullet(z) = \frac{zd}{dz}w(z) \preceq \frac{c_w n^{k-j} p_0^{-1} z^2}{1-p_0^{-1}z}.
  \end{equation}

  The generating function of two-type graphs corresponding to wheels with a marked vertex of type $j$ is dominated by $(c_0+1) w^{\bullet}(z)$, since vertex of type $k$ in such two-type graphs is adjacent to $(c_0+1)$ vertices of type $j$.

  We recall that $T_J(z)$ is the generating function of unlabelled two-type trees (defined in \eqref{eq:T_J}). The generating function of labelled two-type trees with a given $j$-set as root label is given by $T_J(p_0^{-1} c_0^{-1} z)$, since for each vertex of type $k$, there are $\binom{n-j}{k-j} = p_0^{-1} c_0^{-1}$ choices to complete its label from that of its parent, which is of type $j$ and has a $j$-set as label. We have an extra factor of $\binom{n}{j}$ when the root label is not given. We recall (\ref{eq:T-in-Lambert}), where $T_J(z)$ is expressed with the Lambert $W$-function $W(z)$, which satisfies the equation $z=W(z)\exp(W(z))$:
  \[
    T_J(z) = \exp\left( -c_0^{-1} W(-c_0z) \right) - 1.
  \]
  Hence, 
  \begin{equation}\label{eq:misc-1}
    1 + T_J(p_0^{-1} c_0^{-1} z) = \exp(-c_0^{-1}W(-p_0^{-1} z)).
  \end{equation}
  By differentiating the equation $z=W(z)\exp(W(z))$, we have the following expression of $\frac{d}{dz}W(z)$:
  \begin{equation} \label{eq:derivW}
      \frac{d}{dz}W(z) = \frac{1}{\exp(W(z)) (1+W(z))} = \frac{W(z)}{z(1+W(z))}.
  \end{equation}

We can now also compute the derivative of $T_J(z)$ as
\begin{align} 
  \frac{d}{dz}T_J(z) &= \exp(-c_0^{-1}W(-c_0 z)) \cdot (-c_0^{-1}) \cdot \frac{d}{dz}(W(-c_0 z)) \nonumber \\
  &= - \exp(-c_0^{-1}W(-c_0 z)) \frac{W(-c_0 z)}{c_0 z(1+W(-c_0 z))} \label{eq:T-diff}.
\end{align}

We can then give the following dominance relation for the derivative of $w^\bullet(z)$, which will be used later. To simplify the notation, we define 
\[
W_0 = W_0(z) = W(-p_0^{-1} z).
\]
Since $W(z)=z\exp(-W(z))$, we have
\begin{equation} \label{eq:misc-3}
\exp(-W_0) = \frac{W_0}{-p_0^{-1} z}.
\end{equation}
Furthermore, by \eqref{eq:derivW}, we also have
\begin{equation} \label{eq:misc-2}
\frac{d}{dz}W_0 = \frac{W_0}{z (1+W_0)}.
\end{equation}
We thus have
  \begin{align*}
  &\quad \quad \frac{zd}{dz}\left[w^\bullet(z(1+T_J(c_0^{-1}p_0^{-1}z))^{c_0}\right] \stackrel{\eqref{eq:T-in-Lambert}}{=} \frac{zd}{dz}\left[w^\bullet\left(z \exp\left(-W(-p_0^{-1} z)\right) \right) \right] \\
  &\stackrel{\eqref{eq:misc-3}}{=} \frac{zd}{dz}\left[w^\bullet( - p_0 W_0) \right] \; \stackrel{\eqref{eq:boundwbullet}}{\preceq} \; z c_w n^{k-j} p_0^{-1} \frac{d}{dz} \left( \frac{p_0^2 W_0^2}{1+W_0} \right) \\
  &\stackrel{\eqref{eq:misc-2}}{=} \frac{z c_w n^{k-j} p_0 W_0 \left(2 + W_0\right)}{\left( 1 + W_0\right)^2} \cdot \frac{W_0}{z (1+W_0)} = \frac{c_w n^{k-j} p_0 W_0^2 \left(2 + W_0\right)}{\left( 1 + W_0\right)^3}.  
  \end{align*}
  Therefore, we have
  \begin{align} 
  \frac{zd}{dz}\left[w^\bullet(z(1+T_J(c_0^{-1}p_0^{-1}z))^{c_0}\right] &= \frac{c_w n^{k-j} p_0 W_0^2 \left(2 + W_0 \right)}{\left( 1 + W_0\right)^3} \nonumber \\
  &= c_w n^{k-j} p_0 \sum_{i \geq 2} \frac{(-1)^i (i-1) (i+2)}{2} W_0^i, \label{eq:diff-w-bullet}
  \end{align}
where we have used the expansion
\[
\frac{x^2 (2+x)}{(1+x)^3} = \sum_{i \geq 2} \frac{(-1)^i (i-1) (i+2)}{2} x^i.
\]

We now consider the generating functions $B^{\circ,1}(z), B^{\circ,2}(z), B^{\circ,3}(z)$ arising from
Proposition~\ref{prop:decomp-non-hypertree}. For $B^{\circ,1}(z)$, using \eqref{eq:derivW}
and observing that in any
two-type tree the number of non-root type $j$ vertices is $c_0 s$ times the number of type $k$ vertices,
we have
  \begin{align*}
    B^{\circ,1}(z) &\preceq c_0 z \left( \frac{zd}{dz} \binom{n}{j} T_J(p_0^{-1} c_0^{-1} z) \right) (1+T_J(p_0^{-1}c_0^{-1}z))^{c_0} \\
    &\stackrel{\eqref{eq:misc-1}, \eqref{eq:T-diff}}{=} c_0 z^2  \binom{n}{j} (p_0^{-1}c_0^{-1})
    (-\exp(-c_0^{-1}W_0))
    \frac{W_0}{p_0^{-1}z(1+W_0)}\exp(-W_0)\\ 
                &= - z \binom{n}{j} \frac{W_0}{(1+W_0)} \exp(-(1+c_0^{-1}) W_0)) \\
                &= \binom{n}{j} z \sum_{i \geq 1} \sum_{r \geq 0} \frac{(-1)^{i+r} (1+c_0^{-1})^r}{r!} W_0^{i+r}.
  \end{align*}
  The extra factor $z$ in the initial domination comes from the change of the number of vertices of type $k$ in the injection. Therefore, by (\ref{eq:Lambert-power}) and using the fact that $i+r \leq i(r+1)$ for integers $i \geq 1$, $r \geq 0$, we have

  \begin{align*}
    |\cB^{\circ,1}_s| &= [z^s]B^{\circ,1}(z) \leq \binom{n}{j} \sum_{i \geq 1} \sum_{r \geq 0} \frac{(-1)^{i+r+s} (i+r) (1+c_0^{-1})^r (-s+1)^{s-i-r-2}}{p_0^{s-1}(s-1-r-i)!r!} \\
    &=c_0^{s-1} \binom{n}{j} \binom{n-j}{k-j}^{s-1} \sum_{i \geq 1} \sum_{r \geq 0} \frac{i(r+1)(1+c_0^{-1})^r (s-1)^{s-i-r-2}}{(s-1-r-i)!r!} \\
    &=c_0^{s-1} \binom{n}{j} \binom{n-j}{k-j}^{s-1} \sum_{i \geq 1}  \frac{i(s-1)^{s-i-2}}{(s-1-i)!} \sum_{r \geq 0} \frac{(r+1)(1+c_0^{-1})^r (s-1-i)_{(r)}}{(s-1)^r r!} \\
    &\leq \binom{n}{j} \binom{n-j}{k-j}^{s-1} \frac{c_0^{s-1} (s-1)^{s-2}}{(s-1)!} \sum_{i=1}^{s-1}  \frac{i (s-1)_{(i)}}{(s-1)^{i}} \sum_{r \geq 0} \frac{(r+1)2^r}{r!} \\
    &\leq (1+3e^2) \binom{n}{j} \binom{n-j}{k-j}^{s-1} \frac{c_0^{s-1} s^{s-1}}{s!} \sum_{i = 1}^{s-1}  \frac{i s_{(i)}}{s^{i}}.
  \end{align*}
  Now using Lemma~\ref{lem:laplace-app} with $a=1$, for $s \geq 256$ we have
  \[
  |\cB^{\circ,1}_s| \leq (1+3e^2) \binom{n}{j} \binom{n-j}{k-j}^{s-1}  5 \cdot 2^{1/2} \frac{c_0^{s-1} s^{s}}{s!} = O(s n^{j-k}) B_s,
  \]
where the last equality follows from Lemma~\ref{lem:branchingconfigs}.
  
Now for $B^{\circ,2}$ we have
  \begin{align*}
    B^{\circ,2}(z) &\preceq (c_0+1) z \frac{zd}{dz}\left[w^\bullet(z(1+T_J(c_0^{-1}p_0^{-1}z))^{c_0})\right] (1+T_J(c_0^{-1}p_0^{-1}z))^{c_0} \\
                   &\stackrel{\eqref{eq:diff-w-bullet}}{\preceq} 2c_0 z c_w n^{k-j} p_0 \sum_{i \geq 2} \frac{(-1)^i (i-1) (i+2)}{2} W_0^i \exp(-W_0) \\
                   &= 2 z c_0 c_w n^{k-j} p_0 \sum_{i\geq 2} \sum_{r \geq 0} \frac{(-1)^{r+i} (i-1)(i+2)}{2 \cdot r!} W_0^{i+r}.
  \end{align*}
  Therefore, by (\ref{eq:Lambert-power}), again using the bound $i+r \leq i(r+1)$, we have
  \begin{align*}
    |\cB_s^{\circ,2}| &\le [z^s]B^{\circ,2}(z) \leq 2 c_0 c_w n^{k-j} p_0 \sum_{i\geq 2} \sum_{r \geq 0} \frac{(i-1)(i+2)(i+r)(s-1)^{s-r-i-2}}{2 p_0^{s-1} r!(s-r-i-1)!} \\
    &\leq 2 (k-j)!c_0^{s-1} c_w \binom{n-j}{k-j}^{s-1} \sum_{i\geq 2} \frac{i^2(i+1)(s-1)^{s-i-2}}{(s-i-1)!} \sum_{r \geq 0} \frac{(r+1)(s-i-1)_{(r)}}{(s-1)^r r!} \\
    &\leq 2(k-j)!c_0^{s-1} c_w \binom{n-j}{k-j}^{s-1} \frac{(s-1)^{s-2}}{(s-1)!} \sum_{i=2}^{s-1} \frac{2 i^3 (s-1)_{(i)}}{(s-1)^i} \sum_{r \geq 0} \frac{r+1}{r!} \\
    &\leq 4(k-j)! c_0^{s-1}  c_w \binom{n-j}{k-j}^{s-1} \frac{s^{s-1}}{s!} \sum_{i=2}^{s} \frac{i^3 s_{(i)}}{s^i} (1+2e).
  \end{align*}
  Now using Lemma~\ref{lem:laplace-app} with $a=3$ and Lemma~\ref{lem:branchingconfigs}, for $s \geq 2304$ we have
  \[
  |\cB_s^{\circ,2}| \leq 4  (k-j)! (1+2e) c_w \binom{n-j}{k-j}^{s-1} 5\cdot 6^{3/2} \frac{c_0^s s^{s+1}}{s!} = O(s^2 n^{-k}) B_s. \\
  \]

  For $B^{\circ,3}$, using (\ref{eq:diff-w-bullet}), we have
  \begin{align*}
    B^{\circ,3}(z) &\preceq (c_0+1) \frac{zd}{dz}\left[(c_0+1) w^\bullet(z(1+T_J(c_0^{-1}p_0^{-1}z))^{c_0})\right] (1+T_J(c_0^{-1}p_0^{-1}z)) \\
                   &\stackrel{\eqref{eq:diff-w-bullet}}{\preceq} (c_0+1)^2 c_w n^{k-j} p_0 \sum_{i \geq 2} \frac{(-1)^i (i-1) (i+2)}{2} W_0^i \exp(-c_0^{-1} W_0) \\
                   &= (c_0+1)^2  c_w n^{k-j} p_0 \sum_{i\geq 2} \sum_{r \geq 0} \frac{(-1)^{r+i} (i-1)(i+2)}{2c_0^r r!} W_0^{i+r}.
  \end{align*}
  Therefore, by (\ref{eq:Lambert-power}) and the bound $i+r \leq i(r+1)$, we have
  \begin{align*}
    |\cB_s^{\circ,3}| &= [z^s]B^{\circ,3}(z) \le (c_0+1)^2 c_w n^{k-j} p_0 \sum_{i\geq 2} \sum_{r \geq 0} \frac{(i-1)(i+2)(i+r)s^{s-r-i-1}}{2 c_0^r r! (s-r-i)! p_0^{s}} \\
    &\leq \frac{4c_0^2 \cdot 2(k-j)! \cdot c_0^{s-1} c_w s^{s-1}}{s!} \binom{n-j}{k-j}^s \sum_{i\geq 2} \sum_{r \geq 0} \frac{i^2 (i+1) (r+1) s_{(r+i)}}{2 c_0^r r! s^{r+i}} \\
    &\leq \frac{8(k-j)!c_0^{s+1} c_w s^{s-1}}{s!} \binom{n-j}{k-j}^s \sum_{i=2}^{s} \frac{i^3 s_{(i)}}{s^i} \sum_{r \geq 0} \frac{(r+1) (s-i)_{(r)}}{c_0^r r! s^{r}} \\
    &\leq \frac{8(k-j)!c_0^{s+1} c_w s^{s-1}}{s!} \binom{n-j}{k-j}^s \sum_{i=2}^{s} \frac{i^3 s_{(i)}}{s^i} \sum_{r \geq 0} \frac{r+1}{r!} \\
    &= \frac{8(k-j)!(1+2e) c_0^{s+1} c_w s^{s-1}}{s!} \binom{n-j}{k-j}^s \sum_{i=2}^{s} \frac{i^3 s_{(i)}}{s^i}.
  \end{align*}
  Again, using Lemma~\ref{lem:laplace-app} with $a=3$ and Lemma~\ref{lem:branchingconfigs}, for $s \geq 2304$ we have
  \[
  |\cB_s^{\circ,3}| \leq \frac{8(k-j)!(1+2e) c_0^{s+1} c_w \cdot 5 \cdot 6^{3/2} s^{s+1}}{s!} \binom{n-j}{k-j}^s = O(s^2 n^{-j}) B_s.
  \]
  
  Putting everything together,
  \[ |\cB_s^{\circ}| = \left( O (sn^{j-k}) + O(s^2 n^{-k}) + O(s^2 n^{-j})\right)  B_s = \left( O (sn^{j-k})  + O(s^2 n^{-j})\right) B_s. \]
Since $B_s^- = B_s - |\cB_s^{\circ}|$, this completes the proof.
Note that we only need $s \geq 2304$ for all conditions concerning Lemma~\ref{lem:laplace-app} to be fulfilled. \qed

\subsection{Proof of Lemma~\ref{lem:non-hypertree-bound}} \label{sec:wheels:non-hypertree}

Any non-hypertree $j$-component contains a wheel, since this is the only obstacle for a component to be a hypertree. We denote by $N_{\geq s}$\label{var:Ns} the number of non-hypertree components of size at least $s$. We consider the following \tdef{wheel-based branching process}\label{term:wbbranchproc}:  we start with a family of $\ell$ many $k$-sets that is the family of $k$-sets in a possible wheel and check if they exist as hyperedges in $\cH$; if so, we perform $c_0 \ell$ branching process, starting from each $j$-set contained in $k$-sets of the wheel (with possible duplications if some $j$-set belongs to more than two hyperedges or if hyperedges intersect in more than $j$ vertices). Since a wheel of length $\ell$ has at most $c_0\ell$ many $j$-sets, by the same argument as in Lemma~\ref{lem:branchcoupling}, we know that the expected number of non-hypertree $j$-component of size at least $s$ is bounded from above by the expected number of wheel-based branching processes of size at least $s$. Let $N'_{\geq s}$\label{var:N's} be the number of instances of wheel-based branching processes of size at least $s$, starting from every possible family of $k$-sets that could form a wheel. Then we have
\[
  \EX(N_{\geq s}) \leq \EX(N'_{\geq s}).
\]

Let $u_s$\label{par:Us} be the number of possible two-type unicycles with $s$ vertices of type $k$ and with a marked vertex of type $j$. We first bound $u_s$.

\begin{lem} \label{lem:non-hypertree-config}
  For $s \geq 1024$, we have
  \[
    u_s \leq 122 c_0^2 c_w n^{k-j} p_0^{1-s} \frac{s^{s+1/2}}{s!}.
  \]
\end{lem}
\begin{proof}
  Let $U(z)$\label{genfun:U} be the generating function of two-type unicycles with a marked vertex of type $j$, with $z$ indicating their sizes. Recall that $w(z)$ is the generating function of wheels with $z$ indicating the number of hyperedges, defined in \eqref{eq:w}. We recall that $W_0=W_0(z)=W(-p_0^{-1}z)$. Using arguments from the proof of Lemma~\ref{lem:bound-wheels} in Section~\ref{sec:wheels:bound-wheels}, from the correspondence above, we have
  \begin{align*}
    U(z) &\preceq (c_0+1) \frac{zd}{dz} \left[w(z(1+T_J(c_0^{-1}p_0^{-1}z))^{c_0}) \right] \stackrel{\eqref{eq:misc-1}}{\preceq} 2c_0 \frac{zd}{dz} \left[w(z \exp(-W_0)) \right] \\
    &\stackrel{\eqref{eq:misc-3}}{=} 2c_0 \frac{zd}{dz} \left[w(- p_0 W_0  ) \right] \stackrel{\eqref{eq:misc-2}, \eqref{eq:boundwbullet}}{\preceq} \frac{2c_0^2 z c_w n^{k-j} p_0 W_0^2}{1+W_0} \cdot \frac{1}{z (1+ W_0)} \\
    &= \frac{2c_0^2 c_w n^{k-j} p_0 W_0^2}{(1+W_0)^2} = 2c_0^2 c_w n^{k-j} p_0 \sum_{r \geq 2} (-1)^r (r-1) W_0^r.
  \end{align*}
  The initial domination comes from the definition of two-type unicycle and the fact that there are at most $(c_0+1)s$ vertices of type $j$ in a two-type unicycle of size $s$. The last equality follows from the substitution of the Taylor expansion of $x^2 (1+x)^{-2}$ with $x=W_0$.

  We can now use (\ref{eq:Lambert-power}) to estimate $u_s$.
  \begin{align*}
    u_s = [z^s]U(z) &\leq 2c_0^2 c_w n^{k-j} p_0 \sum_{r \geq 2} \frac{r(r-1) s^{s-r-1}}{(s-r)!} p_0^{-s} \\
                    &= 2c_0^2 c_w n^{k-j} p_0^{1-s} \frac{s^{s-3}}{(s-3)!} \left(2 + \sum_{i = 1}^{s-2} \frac{(i+1)(i+2) (s-3)_{(i)}}{s^i} \right) \\
                    &\leq 2c_0^2 c_w n^{k-j} p_0^{1-s} \frac{s^{s-1}}{s!} \left( 2 + \sum_{i=1}^{s} \frac{6i^2 s_{(i)}}{s^i} \right) \\
                    &\leq 2c_0^2 c_w n^{k-j} p_0^{1-s} \frac{s^{s-1}}{s!} \left( 2 + 6 \cdot 20 s^{3/2} \right) \leq 244 c_0^2 c_w n^{k-j} p_0^{1-s} \frac{s^{s+1/2}}{s!}.
  \end{align*}
  The penultimate inequality comes from the application of Lemma~\ref{lem:laplace-app} with $a=2$, which holds for $s \geq 1024$.
\end{proof}

We now estimate $\EX(N'_{\geq s})$. Given a two-type unicycle of size $s$ with a wheel of length $\ell$, the probability for it to appear in the corresponding wheel-based branching process is at most $p^s (1-p)^{sc_0\binom{n-j}{k-j}-sc_0}$, with $p^s$ accounting for the existences of hyperedges, and $(1-p)^{sc_0\binom{n-j}{k-j}-sc_0}$ for the absences of hyperedges. Indeed, since every labelled two-type tree that corresponds to an hypertree with $s'$ vertices of type $k$ has $s'c_0+1$ vertices of type $j$, in all $\ell c_0$ branching processes, $s-\ell$ hyperedges are discovered, meaning that there are in total $(s-\ell)c_0+\ell c_0 = s c_0$ many $j$-sets in the branching process. Then, each $j$-set makes at least $\left(\binom{n-j}{k-j}-1\right)$ queries to $k$-sets, receiving $s-\ell$ affirmatives in total, which makes the number of absences at least
\[
  sc_0\left(\binom{n-j}{k-j}-1\right) - s + \ell \geq sc_0\binom{n-j}{k-j} - s(c_0+1).
\]

Let $P_s$\label{var:Ps} be the number of $j$-sets in wheel-based branching processes of size exactly $s$. We recall that $p = p_0(1-\eps)$. With Lemma~\ref{lem:non-hypertree-config}, we have
\begin{align*}
  \EX(P_s) &\leq u_s p^s (1-p)^{sc_0\binom{n-j}{k-j}-s(c_0+1)} \\
            &\leq \Theta(1) n^{k-j} p_0 (1-\eps)^s \frac{s^{s+1/2}}{s!} (1-p)^{sc_0\binom{n-j}{k-j}-s(c_0+1)} \\
            &\leq \Theta(1) n^{k-j} p_0 (1-\eps)^s e^s \exp \left( -p \left( sc_0\binom{n-j}{k-j}-s(c_0+1)\right) \right).
\end{align*}
We recall that $p_0 = c_0^{-1}\binom{n-j}{k-j}^{-1}$ and $\delta = -\eps -\log(1-\eps)$. Thus, we have
\begin{align*}
  \EX(P_s) &\leq \Theta(1) \exp\left( s\log(1-\eps) + s - (1-\eps)s \right) \exp\left((1-\eps) s (c_0^{-1} + 1)\binom{n-j}{k-j}^{-1}\right) \\
            &\leq \Theta(1) \exp(-s\delta) \exp\left(2 s \binom{n-j}{k-j}^{-1}\right).
\end{align*}

We are now interested in $N'_{\geq s^\circ}$. As we only need $s^\circ \delta \to \infty$, without loss of generality, we may take $s^\circ \delta = o(\log n)$. We can also assume that $s^\circ \leq \binom{n}{k}$, \textit{i.e.}\ the total number of possible hyperedges. Let $\delta^- = \delta - 2\binom{n-j}{k-j}^{-1}$. Since $\delta = \eps^2/2 + O(\eps^3)$, with $\eps^2 n^{k-j}(\log n)^{-1} \to \infty$, we have $\delta^- = (1-o(1))\delta$, therefore $\delta^- > 0$. As each wheel-based branching process with $s$ hyperedges has at least $sc_0$ many $j$-sets, we have
\begin{align*}
  \EX(N'_{\geq s^\circ}) &\leq \sum_{s \geq s^\circ} c_0^{-1} s^{-1} \EX(P_s) \leq \sum_{s \geq s^{\circ}}\Theta(1) s^{-1} \exp(-s\delta^-) \\
                         &\leq \Theta(1) (s^\circ)^{-1} \sum_{s \geq s^{\circ}} \exp(-s\delta^-) \leq \Theta(1) \frac{\exp(-s^\circ \delta^-)}{s^\circ \delta^-}.
\end{align*}

Since $s^\circ (\delta - \delta^-) = s^\circ \binom{n-j}{k-j}^{-1} = o(\delta^{-1} n^{j-k} \log n) = o(1)$, and $s^\circ \delta \to \infty$, we have
\[
  \EX(N_{\geq s^\circ}) \leq \EX(N'_{\geq s^\circ}) \leq \Theta(1) \frac{\exp(-s^\circ \delta) \exp(o(1))}{s^\circ \delta + o(1)} = \Theta(1) \frac{\exp(-s^\circ \delta)}{s^\circ \delta}.
\]
We conclude that $\EX(N_{\geq s^\circ}) \to 0$, using that $s^\circ \delta \to \infty$. By Markov's inequality, whp we have $N_{\geq s^\circ} = 0$, meaning that there is no non-hypertree component in $\cH$ of size at least $s^\circ = \omega(\delta^{-1})$.\qed 
\medskip

\section{Concluding remarks} \label{sec:concrem}

\subsection{The critical window}
We note that our proofs impose two restrictions on $\eps$,
neither of which match the critical window given by $\eps^3 n^j\rightarrow \infty$, conjectured in \cite{CKKgiant}.

The first restriction is that $\eps^4 n^j \to \infty$. We note that the full strength of
this condition is not
needed for the proof of the upper bound (Lemma~\ref{lem:upper-bound}), in which
$\eps^3 n^j \to \infty$ would have been enough.
Rather, the condition is needed for the lower bound (Lemma~\ref{lem:lower-bound-second-moment}), and in particular Lemma~\ref{lem:bound-wheels}.
Heuristically, the discrepancy seems to arise from the fact that $\cB^\circ$
gives a rather bad approximation of a search process which encounters a wheel.
More precisely, in a breadth-first search process, one would explore around a wheel in both
directions and would encounter the same $j$-set or $k$-set again at the same level of the search tree.
By contrast, in $\cB^\circ$ one could encounter a previous $j$-set of $k$-set from many
generations ago, making $\cB^\circ$ a much larger set. Ideally, one would alter the
definition of the two-type branching process to forbid this from happening, but this
modification would make enumeration (particularly the lower bound) significantly more difficult.

The second restriction is that $\eps^2 n^{k-j} (\log n)^{-1} \to \infty$.
The full strength of this condition is used in the upper bound, while in the lower bound
the slightly weaker condition of $\eps^2 n^{k-j} \to \infty$ is needed (\textit{i.e.}\ without
the logarithmic factor). It is not clear whether these conditions are simply an artefact
of our proof methods or, particularly since the polynomial part appears in both upper
and lower bounds, a necessary restriction.
Further study is needed to clarify the situation. 

\subsection{Global structure}
In this paper we investigated the structure of the largest components,
but it would be interesting to characterise the structure of the \emph{whole} subcritical hypergraph.
To this end, one would need to investigate wheels more carefully.

Wheels can be classified by the size
of their \emph{centre}, \textit{i.e.}\ the set of vertices shared by every hyperedge in the same wheel.
If the centre of the wheel contains $i$ vertices, we call it an $i$-wheel.

Heuristically, the behaviour of the hypergraph with respect to $i$-wheels is fundamentally
dependent on $i$. In particular, at what probability we can expect $i$-wheels to appear
and how frequently changes as $i$ ranges from $0$ to $j-1$ -- those with a larger
centre appear first. Furthermore,
the length of a typical $i$-wheel seems to be constant if $i\ge 1$, but
logarithmic in $n$ for $i=0$, which makes the analysis of this case rather different.

It would be interesting to study these results further and rigorously prove
these heuristics, as well as tackling further questions, such as
\begin{itemize}
\item How many $i$-wheels are there in total, and of which length?
\item How many $j$-sets are in components containing $i$-wheels?
\item What is the size of the largest component that contains an $i$-wheel?
\item Do any components contain more than one wheel?
\end{itemize}

\subsection{Symmetry phenomenon}

As mentioned in the introduction, random graphs display a symmetry phenomenon around the phase transition:
the subcritical random graph with $p=(1-\eps)p_0$ has approximately the same distribution as the
supercritical random graph with $p=(1+\eps)p_0$ with the giant component removed, in particular
regarding the orders of components and their structure (they are all trees or unicyclic).
More generally, the same is true in $k$-uniform hypergraphs for the case $j=1$
(see for example~\cite{BCOK14}, Lemma~15).

However, the analogous result for any $j\ge 2$ has not yet been proved. The order of the giant component
was asymptotically determined in~\cite{CKKgiant}, but in contrast to the case $j=1$,
the distribution of the remaining hypergraph fundamentally depends not only on the number of $j$-sets
in the giant component, but also on how they are distributed across the vertices, making this
case rather more challenging.

\bibliographystyle{plain}
\bibliography{References}

\begin{thebibliography}{10}

\bibitem{AjaziNapolitanoTurova17}
F.~Ajazi, G.~M. Napolitano, and T.~Turova.
\newblock Phase transition in random distance graphs on the torus.
\newblock {\em J. Appl. Probab.}, 54(4):1278--1294, 2017.

\bibitem{AndriamaRavelomanana05}
T.~Andriamampianina and V.~Ravelomanana.
\newblock Enumeration of connected uniform hypergraphs.
\newblock In {\em Proceedings of 17th International Conference on Formal Power
  Series and Algebraic Combinatorics (FPSAC 2005)}, pages 387 -- 398, 2005.

\bibitem{BehrischCojaOghlanKang10b}
M.~Behrisch, A.~Coja-Oghlan, and M.~Kang.
\newblock The order of the giant component of random hypergraphs.
\newblock {\em Random Structures \& Algorithms}, 36(2):149--184, 2010.

\bibitem{BCOK14}
M.~Behrisch, A.~Coja-Oghlan, and M.~Kang.
\newblock Local limit theorems for the giant component of random hypergraphs.
\newblock {\em Combinatorics, Probabability and Computing}, 23(3):331--366,
  2014.

\bibitem{Bollobas84}
B.~Bollob{\'a}s.
\newblock The evolution of random graphs.
\newblock {\em Trans. Amer. Math. Soc.}, 286(1):257--274, 1984.

\bibitem{BollobasBook}
B.~Bollob{\'a}s.
\newblock {\em Random graphs}, volume~73 of {\em Cambridge Studies in Advanced
  Mathematics}.
\newblock Cambridge University Press, Cambridge, second edition, 2001.

\bibitem{BollobasJansonRiordan07}
B.~Bollob{\'a}s, S.~Janson, and O.~Riordan.
\newblock The phase transition in inhomogeneous random graphs.
\newblock {\em Random Structures \& Algorithms}, 31(1):3--122, 2007.

\bibitem{BollobasRiordan03}
B.~Bollob\'{a}s and O.~Riordan.
\newblock Mathematical results on scale-free random graphs.
\newblock In {\em Handbook of graphs and networks}, pages 1--34. Wiley-VCH,
  Weinheim, 2003.

\bibitem{BollobasRiordan12c}
B.~Bollob{\'a}s and O.~Riordan.
\newblock Asymptotic normality of the size of the giant component in a random
  hypergraph.
\newblock {\em Random Structures \& Algorithms}, 41(4):441--450, 2012.

\bibitem{BollobasRiordan17}
B.~Bollob{\'a}s and O.~Riordan.
\newblock Exploring hypergraphs with martingales.
\newblock {\em Random Structures \& Algorithms}, 50(3):325--352, 2017.

\bibitem{CKKgiant}
O.~Cooley, M.~Kang, and C.~Koch.
\newblock The size of the giant high-order component in random hypergraphs.
\newblock {\em Random Structures \& Algorithms}, 53(2):238--288, 2018.

\bibitem{CooleyKangPerson18}
O.~{Cooley}, M.~{Kang}, and Y.~{Person}.
\newblock {Largest components in random hypergraphs}.
\newblock {\em Combinatorics, Probability and Computing}, pages 1--22, 2018.

\bibitem{dePanafieu15}
\'{E}. de~Panafieu.
\newblock Phase transition of random non-uniform hypergraphs.
\newblock {\em Journal of Discrete Algorithms}, 31:26--39, 2015.

\bibitem{ErdosRenyi60}
P.~Erd{\H{o}}s and A.~R{\'e}nyi.
\newblock On the evolution of random graphs.
\newblock {\em Magyar Tud. Akad. Mat. Kutat\'o Int. K\"ozl.}, 5:17--61, 1960.

\bibitem{FlajoletSedgewickBook}
P.~Flajolet and R.~Sedgewick.
\newblock {\em Analytic combinatorics}.
\newblock Cambridge University Press, Cambridge, 2009.

\bibitem{Janson08}
S.~Janson.
\newblock The largest component in a subcritical random graph with a power law
  degree distribution.
\newblock {\em Ann. Appl. Probab.}, 18(4):1651--1668, 2008.

\bibitem{KahlePittel16}
M.~Kahle and B.~Pittel.
\newblock Inside the critical window for cohomology of random {$k$}-complexes.
\newblock {\em Random Structures \& Algorithms}, 48(1):102--124, 2016.

\bibitem{KaronskiLuczak96}
M.~Karo{\'n}ski and T.~{\L}uczak.
\newblock Random hypergraphs.
\newblock In {\em Combinatorics, {P}aul {E}rd\H{o}s is eighty, {V}ol.\ 2
  ({K}eszthely, 1993)}, volume~2 of {\em Bolyai Soc. Math. Stud.}, pages
  283--293. J\'anos Bolyai Math. Soc., Budapest, 1996.

\bibitem{KaronskiLuczak97}
M.~Karo{\'n}ski and T.~{\L}uczak.
\newblock The number of connected sparse edged uniform hypergraphs.
\newblock {\em Discrete Math.}, 171:153--168, 1997.

\bibitem{KaronskiLuczak02}
M.~Karo{\'n}ski and T.~{\L}uczak.
\newblock The phase transition in a random hypergraph.
\newblock {\em J. Comput. Appl. Math.}, 142(1):125--135, 2002.
\newblock Probabilistic methods in combinatorics and combinatorial
  optimization.

\bibitem{LinialMeshulam06}
N.~Linial and R.~Meshulam.
\newblock Homological connectivity of random 2-complexes.
\newblock {\em Combinatorica}, 26(4):475--487, 2006.

\bibitem{LinialPeled16}
N.~Linial and Y.~Peled.
\newblock On the phase transition in random simplicial complexes.
\newblock {\em Ann. of Math. (2)}, 184(3):745--773, 2016.

\bibitem{Luczak90}
T.~{\L}uczak.
\newblock Component behavior near the critical point of the random graph
  process.
\newblock {\em Random Structures \& Alg.}, 1(3):287--310, 1990.

\bibitem{MeshulamWallach08}
R.~Meshulam and N.~Wallach.
\newblock Homological connectivity of random {$k$}-dimensional complexes.
\newblock {\em Random Structures \& Algorithms}, 34(3):408--417, 2009.

\bibitem{Poole15}
D.~Poole.
\newblock On the strength of connectedness of a random hypergraph.
\newblock {\em Electron. J. Combin.}, 22(1):Paper 1.69, 16, 2015.

\bibitem{RavelomananaRijamamy2006}
V.~Ravelomanana and A.~L. Rijamamy.
\newblock Creation and growth of components in a random hypergraph process.
\newblock {\em 12th International Computing and Combinatorics Conference
  (COCOON)}, pages 350--359, 2006.

\bibitem{SchmidtShamir85}
J.~Schmidt-Pruzan and E.~Shamir.
\newblock Component structure in the evolution of random hypergraphs.
\newblock {\em Combinatorica}, 5(1):81--94, 1985.

\bibitem{Turova11}
T.~S. Turova.
\newblock The largest component in subcritical inhomogeneous random graphs.
\newblock {\em Combinatorics, Probability and Computing}, 20(1):131--154, 2011.

\bibitem{Wormald99}
N.~C. Wormald.
\newblock Models of random regular graphs.
\newblock In {\em Surveys in combinatorics, 1999 ({C}anterbury)}, volume 267 of
  {\em London Math. Soc. Lecture Note Ser.}, pages 239--298. Cambridge
  University Press, Cambridge, 1999.

\end{thebibliography}

\newpage

\appendix

\section{Glossary}\label{ap:glossary}

For the reader's convenience, we include a glossary of some of the most important terminology and notation defined in the paper.

\vspace{0.3cm}

\subsection{Terminology}

\vspace{0.3cm}

\begin{tabular}{|l|m{6cm}|l|}
\hline
\textbf{Term}${\phantom{\bigg(\bigg)}}$
 & \textbf{Informal description} & \textbf{First defined}\\
\hline
Hypertree & Hypergraph analogue of a tree & p.~\pageref{term:hypertree}${\phantom{\bigg(\bigg)}}$ \\
\hline
Wheel & Hypergraph analogue of cycle & p.~\pageref{term:wheel}${\phantom{\bigg(\bigg)}}$ \\
\hline
Two-type graph & Bipartite graph, vertices type $j$ and $k$ with additional structure & p.~\pageref{term:twotypegr}${\phantom{\bigg(\bigg)}}$ \\
\hline
Labelled two-type graph & Vertices labelled by $j$- and $k$-sets & p.~\pageref{term:labtwotypgr}${\phantom{\bigg(\bigg)}}$ \\
\hline
Two-type tree & Acyclic two-type graph & p.~\pageref{term:twotyptree}${\phantom{\bigg(\bigg)}}$\\
\hline
Two-type unicycle & Unicyclic two-type graph & p.~\pageref{term:twotypunic}${\phantom{\bigg(\bigg)}}$\\
\hline
Component search process & Breadth-first search process on $j$-connected components &p.~\pageref{term:compsearchproc}${\phantom{\bigg(\bigg)}}$\\
\hline
Two-type branching process & Upper coupling for the component search process &p.~\pageref{term:branchproc}${\phantom{\bigg(\bigg)}}$\\
\hline
Wheel-based branching process & Modified branching process starting from (potential) wheels &p.~\pageref{term:wbbranchproc}${\phantom{\bigg(\bigg)}}$\\
\hline
\end{tabular}

\vspace{1cm}
\subsection{Classes}

\vspace{0.3cm}

\begin{tabular}{|l|m{9cm}|l|}
\hline
\textbf{Class}${\phantom{\bigg(\bigg)}}$ & \textbf{Informal description} & \textbf{First defined}\\
\hline
$\cT$ & Two-type branching process (with labels) & p.~\pageref{class:branchproc}${\phantom{\bigg(\bigg)}}$\\
\hline
$\cB$ & Set of possible instances of $\cT$ & p.~\pageref{class:B}${\phantom{\bigg(\bigg)}}$\\
\hline
$\cB^-$ & Subset corresponding to hypertrees & p.~\pageref{class:Bmin}${\phantom{\bigg(\bigg)}}$ \\
\hline
$\cB^\circ$ & $\cB\setminus \cB^-$, \textit{i.e.}\ instances containing a wheel & p.~\pageref{class:Bcirc}${\phantom{\bigg(\bigg)}}$\\
\hline
\end{tabular}

\vspace{1cm}

\subsection{Random variables}

\vspace{0.3cm}

\begin{tabular}{|l|m{9cm}|l|}
\hline
\textbf{Variable}${\phantom{\bigg(\bigg)}}$ & \textbf{Informal description} & \textbf{First defined}\\
\hline
$L_i$ & Size of $i$-th largest component & p.~\pageref{var:Li}${\phantom{\bigg(\bigg)}}$\\
\hline
$D_s$ & \# components in $\cH$ of size $s$ & p.~\pageref{var:Ds}${\phantom{\bigg(\bigg)}}$\\
\hline
$S_+$ & \# $j$-sets in components of size between $s_*$ and $s^*$ & p.~\pageref{var:S+}${\phantom{\bigg(\bigg)}}$\\
\hline
$C_s$ & \# $j$-sets in hypertree components in $\cH$ of size $s$ & p.~\pageref{var:Cs}${\phantom{\bigg(\bigg)}}$\\
\hline
$N_s$ & \# non-hypertree components of size $s$ & p.~\pageref{var:Ns}${\phantom{\bigg(\bigg)}}$\\
\hline
$N_s'$ & \# wheel-based branching processes of size $s$, when starting a branching process from each possible wheel & p.~\pageref{var:N's}${\phantom{\bigg(\bigg)}}$\\
\hline
$P_s$ & Total \# $j$-sets in wheel-based branching processes of size $s$, when starting a branching process from each possible wheel & p.~\pageref{var:Ps}${\phantom{\bigg(\bigg)}}$\\
\hline
$R_s$ & Total \# type $j$ vertices in instances of $\cT$ of size $s$ among $\binom{n}{j}$ independent instances & p.~\pageref{var:Rs}${\phantom{\bigg(\bigg)}}$\\
\hline
\end{tabular}

\vspace{1cm}

\subsection{Parameters}

\vspace{0.3cm}

\begin{tabular}{|l|m{8cm}|l|}
\hline
\textbf{Parameter}${\phantom{\bigg(\bigg)}}$ & \textbf{Informal description} & \textbf{First defined}\\
\hline
$c_0$ & $\binom{k}{j}-1$ & p.~\pageref{par:c0andp0}${\phantom{\bigg(\bigg)}}$\\
\hline
$p_0$ & $c_0^{-1}\binom{n-j}{k-j}^{-1}$ & p.~\pageref{par:c0andp0}${\phantom{\bigg(\bigg)}}$\\
\hline
$\delta$ & $-\eps - \log(1-\eps)= \frac{\eps^2}{2} + O(\varepsilon^2)$ & p.~\pageref{par:delta}${\phantom{\bigg(\bigg)}}$\\
\hline
$\lambda$ & $\eps^3 \binom{n}{j}$ & p.~\pageref{par:lambda}${\phantom{\bigg(\bigg)}}$\\
\hline
$u_s$ & \# possible two-type unicycles, $s$ vertices of type $k$, one marked vertex of type $j$ & p.~\pageref{par:Us}${\phantom{\bigg(\bigg)}}$ \\
\hline
$w_\ell$ & \# possible wheels of length $\ell$ & p.~\pageref{par:wl}${\phantom{\bigg(\bigg)}}$ \\
\hline
$w^*_\ell(a_1,\ldots,a_\ell)$ & \# possible wheels of length $\ell$ where $J_i$ freezes $a_i$ vertices (with fixed starting $j$-set and orientation) & p.~\pageref{par:wlstar}${\phantom{\bigg(\bigg)}}$ \\
\hline
\end{tabular}

\vspace{1cm}

\subsection{Generating functions}

\vspace{0.3cm}

\begin{tabular}{|l|m{8.5cm}|l|}
\hline
\textbf{Function}${\phantom{\bigg(\bigg)}}$  & \textbf{Associated class} & \textbf{First defined}\\
\hline
$T_J(z)$ & Unlabelled two-type trees rooted at a vertex $j$ of type $k$ & p.~\pageref{genfun:TJ}${\phantom{\bigg(\bigg)}}$  \\
\hline
$w(z)$ & Wheels & p.~\pageref{eq:w}${\phantom{\bigg(\bigg)}}$  \\
\hline
$w^\bullet(z)$ & Wheels with a marked vertex of type $k$ & p.~\pageref{genfun:wbullet}${\phantom{\bigg(\bigg)}}$ \\
\hline
$U(z)$ & Two-type unicycles with a marked vertex of type $j$ & p.~\pageref{genfun:U}${\phantom{\bigg(\bigg)}}$ \\
\hline
\end{tabular}

\end{document}